%% file: article.tex
\documentclass[12pt]{article}

\usepackage{graphicx}
\usepackage[textwidth=78ex,centering]{geometry}
\usepackage{amsmath}
\usepackage{amsthm}
\usepackage{dsfont}

\theoremstyle{plain}
\newtheorem{maintheorem}{Theorem}

\swapnumbers
\newtheorem{theorem}{Theorem}[section]
\newtheorem{lemma}[theorem]{Lemma}
\newtheorem{lemmadef}[theorem]{Lemma and definition}
\newtheorem{cor}[theorem]{Corollary}
\newtheorem{cordef}[theorem]{Corollary and definition}
\newtheorem{prop}[theorem]{Proposition}
\newtheorem{propdef}[theorem]{Proposition and defintion}
\newtheorem{observation}[theorem]{Observation}
\theoremstyle{definition}
\newtheorem{definition}[theorem]{Definition}
\newtheorem{notation}[theorem]{Notation}

\DeclareMathOperator{\maps}{\rightarrow}     % map
\DeclareMathOperator{\im}{\mathrm{im}}       % image
\DeclareMathOperator{\cham}{\mathrm{Ch}}     % chambers
\DeclareMathOperator{\vt}{\mathrm{vt}}       % vertices
\DeclareMathOperator{\simp}{\mathcal{S}}     % simplices
\DeclareMathOperator{\lk}{Lk}                % link
\DeclareMathOperator{\st}{St}                % star
\DeclareMathOperator{\op}{op}                % opposite
\DeclareMathOperator{\proj}{proj}            % projection
\DeclareMathOperator{\restr}{\mathcal{R}}    % restriction
\DeclareMathOperator{\height}{ht}            % height
\DeclareMathOperator{\conv}{Conv}            % convex hull
\DeclareMathOperator{\opp}{\mathrm{Opp}}     % antipodal points
\DeclareMathOperator{\defcl}{\mathcal{C}}
\DeclareMathOperator{\defclp}{\mathcal{C}^\dagger}
\DeclareMathOperator{\defclpp}{\mathcal{C}^\ddagger}
\newcommand{\hor}{\mathrm{hor}}
\newcommand{\ver}{\mathrm{ver}}
\newcommand{\projspace}[1]{\mathds{P}_{\! #1}}    % projective space
\newcommand{\finitefield}[1]{\mathds{F}_{\! #1}}  % finite field
\DeclareMathOperator{\aut}{\mathrm{Aut}}          % automorphisms
\DeclareMathOperator{\rank}{rk}        % rank
\newcommand{\abs}[1]{\lvert#1\rvert}   % absolute value
\renewcommand{\theta}{\vartheta}
\renewcommand{\epsilon}{\varepsilon}
\renewcommand{\phi}{\varphi}
\begin{document}
    \title{Spherical subcomplexes of spherical buildings}

    \author{Bernd Schulz}
    \date{July 2010}

    \maketitle

    \begin{abstract}
        \noindent Let $\Delta$ be a thick, spherical building equipped with its natural CAT(1) metric and let $M$ be a proper, convex subset of $\Delta$. If $M$ is open or if $M$ is a closed ball of radius $\pi/2$, then $\Lambda$, the maximal subcomplex supported by $\Delta\setminus M$, is $\dim\Lambda$--spherical and non contractible.
    \end{abstract}

    \input{introduction}
    \input{recall}
    \input{cccomp}
    \input{hcomp}
    \input{filter}
    \input{subcomp}
    \input{sphproof}

    \input{address}
\end{document}

%% file: introduction.tex
%
%   introduction.tex
%

\noindent In the late 80's, highly connected subcomplexes of spherical buildings gained attention, as P.~Abramenko in \cite{ab1} and H.~Abels in \cite{al1} independently determined the finiteness length of $\mathrm{SL}_n(\finitefield{q}[t ])$ provided that $q$ is big enough compared with $n$. Later on, in \cite{ab4}, P.~Abramenko generalized the result to absolutely almost simple classical $\finitefield{q}$-groups of positive rank over $\finitefield{q}[t]$.

The proofs used the action of these groups on an Euclidean building $X$ and the existence of a filtration of $X$ with highly connected relative links. Since links in Euclidean buildings are spherical buildings, the search for highly connected subcomplexes of spherical buildings was a key problem. Specifically the restrictions on $q$ in the above results have been made to get the desired connectivity properties of the relative links. H.~Behr's characterization of finitely generated and finitely presented S-arithmetic groups over function fields \cite{bh2}, as well as the last chapter of \cite{ab1} suggested that a generalization of the above results without restrictions on $q$ should be possible. Hence, finding a suitable class of subcomplexes of spherical buildings is a necessary step in carrying out the geometric part of finiteness proofs for S-arithmetic groups over function fields. Suitable means: The subcomplexes admit the action of a parabolic subgroup of $\aut(\Delta)$, the fundamental domain for this action is small, they are highly connected regardless of type or thickness of the underlying building, and they have a handy description. In the current paper such a class of subcomplexes will be defined. The results are:

Let $M$ be a subset with property (P) of a geometrically realized spherical building. Then the maximal subcomplex contained in $M$ is called a (P) supported subcomplex.

\medskip\noindent
{\bf Theorem \ref{theorem-a}.} {\it Non empty, closed, coconvex supported subcomplexes of spherical buildings are homotopy Cohen--Macaulay. They are non contractible for thick buildings of dimension at least one.}\medskip

\noindent The subcomplex supported by the complement of a closed (resp.\ an open) ball with radius $\pi/2$ is called an open (resp.\ a closed) hemisphere complex. Note, that closed hemisphere complexes are closed, coconvex supported subcomplexes.

\medskip\noindent
{\bf Theorem \ref{theorem-b}.} {\it Open hemisphere complexes of thick, spherical buildings are non contractible homotopy Cohen--Macaulay complexes.}\medskip

\noindent The current paper is a shortened version of my thesis \cite{schu}. In \cite{buwo1} the upper bound of the finiteness length of S-arithmetic groups over function fields has been determined by K.-U.~Bux and K.~Wortman without local topological arguments, but for rank-1-groups the lower bound was derived by the same authors using theorems \ref{theorem-a} and \ref{theorem-b} \cite{buwo2}. Furthermore the finiteness length of absolutely almost simple $\finitefield{q}$-groups of positive rank over $\finitefield{q}[ t ]$ has been determined by K.-U.~Bux, R.~Gramlich, and S.~Witzel using theorem~\ref{theorem-b} and a generalized version of theorem \ref{theorem-a} \cite{bgw}.
\medskip

\noindent {\bf Acknowledgement.} I would like to thank Peter~Abramenko for his support on writing this paper and also Kai-Uwe~Bux for commenting on earlier versions. 

%% file: recall.tex
%
%   recall.tex
%
%   23.06.2010
%

\section{Notations, conventions and recalls}

\subsection{Simplicial complexes}

We identify simplicial complexes with their geometric realization. The sets of vertices and simplices of a simplicial complex $X$ will be denoted by $\vt(X)$ and $\simp(X)$, respectively. Simplices are open (in their closure). $\st\sigma$ denotes the star of a simplex $\sigma$. The star of a point is the star of the simplex carrying that point. The link $\lk\sigma$ of a simplex $\sigma\in\simp(X)$ is the subcomplex of $X$ whose simplices $\tau$ are disjoint from $\sigma$ but the upper bound $\sigma\cup\tau$ exists. We will write $\lk_X\sigma$ and $\st_X\sigma$ if $X$ is a subcomplex. The join of two simplicial complexes $X$, $Y$ will be denoted by $X*Y$.

\subsection{Construction of spherical simplicial complexes}

We adopt the definitions from D.~Quillen \cite[Section~8]{q1}. A simplicial complex is n--spherical (or spherical) if it is n--dimensional and (n-1)--connected. By convention non empty complexes are (-1)--connected. The empty complex is (-1)--dimensional and (-2)--connected. A simplicial complex $X$ is said to be homotopy Cohen--Macaulay if $\lk\sigma$ is ($\dim X-\dim\sigma-1$)--spherical for every simplex $\sigma\in\simp(X)$ (including $\emptyset\in\simp(X)$).

A commonly used way to show connectivity properties of simplicial complexes is to build these complexes from complexes with known connectivity properties. An overview of the necessary methods can be found in \cite{bj1}. A common method is to build up joins, because the joins of spherical complexes are known to be spherical (\cite[Korollar zu Bemerkung~6]{ab1}, see also \cite{vo} proof of 1.1). The following proposition is also a standard tool. It is a consequence of the Hurewicz isomorphy theorem \cite[p.~398]{sp}, the Mayer-Vietoris sequence of reduced homology, and Van Kampens theorem \cite[6.4.3]{hw}.

\begin{prop}\label{mainconstr}
    Let $I$ be an index set. Let $X$ and $Y_i$ for $i\in I$ be subcomplexes
    of a simplicial complex $Z=X\cup\bigcup_{i\in I}Y_i$. Assume
    $Y_i\cap Y_j\subseteq X$ for all $i,j\in I$ with $i\neq j$.
    \begin{itemize}
        \item[{\rm a)}]
        If $X$ and $Y_i$ are n-connected and $X\cap Y_i$ is (n-1)--connected
        for all $i\in I$, then $Z$ is n-connected.
        \item[{\rm b)}]
        If $Z$ and $X\cap Y_i$ are n--connected for all $i\in I$,
        so is $X$.
    \end{itemize}
\end{prop}

\noindent To show $n$--connectedness of a connected simplicial
complex, it is sufficient to prove that every finite subcomplex is
contained in a $n$--connected subcomplex, because continuous
images of spheres and balls are contained in finite subcomplexes
(by compactness). Since the metric topology and the weak topology
coincide on finite subcomplexes, one is allowed to use metric
topology.

\subsection{Spherical buildings}

Geometrically realized spherical buildings $\Delta$ admit a unique metrication, invariant under automorphisms, such that apartments are isometric to the $\dim\Delta$--dimensional standard sphere \cite[II.10 Theorem 10A.4]{brha}. We will denote the corresponding canonical metric by $d$ (or by $d_\Delta$ if it is necessary to avoid confusions). ($\Delta,d$) is complete. Isomorphisms of apartments induce isometries and $d$ is a length metric. According to \cite[Proposition~12.18]{abbr}, retractions onto apartments are distance decreasing. We record the precise statement.

\begin{prop}\label{building:retraction}
    The retraction $\rho=\rho_{\Sigma,C}:\Delta\rightarrow\Sigma$
    onto $\Sigma$ centered at $C$ is distance decreasing for every
    apartment $\Sigma$ and every chamber $C$ of $\Sigma$, i.e.,
    $d(\rho(x),\rho(y))\leq d(x,y) $ for all $x,y\in\Delta$.
    Equality holds if $x\in\overline{C}$.
\end{prop}

\begin{notation}
    For $x,y\in\Delta$ we put $[x,y] = \{ z\in\Delta\mid
    d(x,y)=d(x,z)+d(z,y)\}$. As usual, we replace a square bracket
    by a round bracket if the corresponding endpoint is left out.
\end{notation}

\begin{prop}\label{building:segments}
    If $d(x,y)<\pi$, then $[x,y]$ lies in any apartment that contains $x$ and $y$. Therefore $[x,y]$ is the unique segment joining $x$ and $y$. If  $d(x,y)=\pi$ then $[x,y]$ is the union of apartments containing $x$ and $y$.
\end{prop}

\noindent According to \cite[II.1~Proposition~1.4~(1)]{brha}, there are deformations along geodesic segments for spherical buildings.

\begin{propdef}\label{building:deformation}
    For $x,y\in\Delta$ with $d(x,y)<\pi$ and $t\in[0,1]$ let $r_\Delta(x,y,t)\in[x,y]$ be the point defined by $d(x,r_\Delta(x,y,t))=td(x,y)$. The map
    \begin{equation*}
        r_\Delta:\{(x,y)\in\Delta\times\Delta\mid d(x,y)<\pi\} \times[0,1] \maps\Delta;\thickspace (x,y,t)\mapsto r_\Delta(x,y,t)
    \end{equation*}
    is continuous with respect to the metric topology.
\end{propdef}

\noindent By the uniqueness of $d$, the apartments of $\Delta$ are
spheres, triangulated by the hyperplanes of a finite essential
reflection group, since finite Coxeter complexes can be realized
this way. It is clear that roots are closed hemispheres and that
walls are the corresponding equators. Hence, $\sigma,\tau\in\simp(\Delta)$ are opposite if and only if there are points $x\in \sigma$ and $y\in\tau$ with $d(x,y)=\pi$. Although it would be immediate to call such points opposite, two points at distance $\pi$ are called antipodal.

\begin{notation}
    For a point $x\in\Delta$ the set of its antipodal points
    will be denoted by $\opp(x)$. Furthermore we denote
    $\opp^*(x)=\opp(x)\cup\{x\}$.
\end{notation}

\noindent A subset of $\Delta$ is convex if and only if it is
$\pi$--convex, this means, if and only if it contains the joining segments for every pair of non antipodal points out of it. The complement of a convex set is said to be coconvex.

By \cite[Theorem~2.19]{ti1} a subcomplex of $\Delta$ is convex in the sense of \cite[1.5]{ti1} if and only if its intersection with any apartment is an intersection of roots, within if and only if its intersection with any apartment is $\pi$--convex. Then it is also convex as a subset, since by \ref{building:segments} a subset is convex if and only if its intersection with any apartment is convex.

For a set of simplices $M\subseteq\simp(\Delta)$ we denote the full convex hull in the sense of \cite[1.5]{ti1} by $\conv(M)$. Note, that this differs in general from the metric convex hull.

\begin{propdef}\label{building:geodesicproj}
    Let $x\in\Delta$ be a point. For $y\in\Delta\setminus\opp^*(x)$ exists a unique point $p_xy\in\partial\st x$ such that $p_xy\in[x,y]$ or $y\in[x,p_xy]$. The geodesic projection
    \begin{equation*}
        p_x:\Delta\setminus\opp^*(x)\longrightarrow\partial\st x;\thickspace
        y \longmapsto p_xy
    \end{equation*}
    with center $x$ onto the boundary of $\st x$ is continuous with
    respect to the metric topology.
\end{propdef}
\begin{proof}
    The map $p_x$ is well defined by \ref{building:segments}. The continuity follows from \ref{building:retraction} since the restriction $p_x|_\Sigma$ is continuous for any apartment $\Sigma$ that contains $x$.
\end{proof}

\noindent For a simplex $\sigma\in\simp(\Delta)$, the projection to the star of $\sigma$ in the sense of \cite[2.30]{ti1} will be denoted by $\proj_\sigma$. By definition, $\proj_\sigma\tau$ is the maximal simplex of $\st\sigma\cap\conv(\sigma,\tau)$. The geodesic projection and the combinatorial projection are related by the following lemma.

\begin{lemma}\label{building:projection}
    Let $x\in\Delta$ and $y\in\Delta\setminus\opp^*(x)$ be points. Let $\sigma$ and $\tau$ be the simplices carrying $x$ and $y$, respectively. Then $(x,p_xy)$ is contained in $\proj_\sigma\tau$.
\end{lemma}

\noindent For $x\in\Delta$ and $y,z\in\Delta\setminus\opp^*(x)$ let
$\angle_x(y,z)$ denote the angle of the triangle $(x,y,z)$ at
$x$. Since links are spherical buildings and the canonical metric is unique, we get following lemma from \cite[Proposition~2.3~(2)]{cl}.

\begin{lemma}\label{building:link-metric}
    The canonical metric on the link of a vertex $x$ is given by
    $\angle_x$.
\end{lemma}

\noindent The spherical law of cosines \cite[1.2~Proposition~2.2]{brha} relates the length of a side in a spherical triangle to its opposite angle. Using \ref{building:retraction}, geodesic projection, and additionally \ref{building:segments} for the "only if"-part, one gets:

\begin{prop}[{\bf Spherical law of cosines}]\label{building:sph-cos-law}
    Let $x\in\Delta$ be a point and let $y,z\in\Delta\setminus\opp^*(x)$. Then:
    \begin{equation*}
        \cos d(y,z)\leq\cos d(x,y)\cos d(x,z)+\sin d(x,y)\sin
        d(x,z)\cos\angle_x(y,z)
    \end{equation*}
    Equality holds if and only if $x$, $y$, and $z$ are contained in an apartment.
\end{prop}

\noindent If $\Delta=\Delta_1*\Delta_2$ is a reducible spherical building, then $\Delta$ is a spherical join, i.e. the inclusions $\Delta_k\subset\Delta$ are isometric embeddings and the distance of points lying in different factors is $\pi/2$. Hence, chambers of reducible spherical buildings contain points at distance $\pi/2$.

\begin{lemma}\label{building:edgelength}
    The length of edges joining two vertices does not exceed $\pi/2$.
\end{lemma}
\begin{proof}
    We use induction on $\dim\Delta$. The case $\dim\Delta=1$ is clear. Suppose $\dim\Delta>1$. Let $x, y,$ and $z$ be vertices of a common chamber. By the induction hypothesis and \ref{building:link-metric} the angles $\angle_x(y,z)$, $\angle_y(x,z)$ and $\angle_z(x,y)$ of the triangle $(x,y,z)$ are not obtuse. We therefore get the assertion, since the edges of a spherical triangle without obtuse angles can not be longer than $\pi/2$.
\end{proof}

\begin{cor}\label{building:chamberdiameter}
    The diameter of closed chambers does not exceed $\pi/2$.
\end{cor}

\begin{prop}\label{building:join}
    Let $C$ be a chamber of $\Delta$ and $x,y\in\overline{C}$. If $d(x,y)=\pi/2$, then $\Delta=\Delta_x*\Delta_y$ is a spherical join, s.t.\ $x\in\Delta_x$ and $y\in\Delta_y$. I.e. a spherical building is reducible if and only if there is a chamber containing points at distance $\pi/2$.
\end{prop}
\begin{proof}
    The assertion follows immediately from the Buekenhout product theorem \cite[Theorem~7.3]{bue}, once we show that $C$ has two complementary faces $\sigma$ carrying $x$ and $\tau$ carrying $y$ such that $d(u,v)=\pi/2$, for all $u\in\vt(\sigma)$ and $v\in\vt(\tau)$.

    Let $u$ be some vertex of $C$. By \ref{building:chamberdiameter} we get $d(x,u)\leq\pi/2$ and $d(y,u)\leq\pi/2$. Furthermore $\angle_u(x,y)$ does not exceed $\pi/2$ by \ref{building:chamberdiameter} and \ref{building:link-metric}. Hence, using the spherical law of cosines on the triangle $(x,u,y)$ we obtain $d(x,u)=\pi/2$ or $d(y,u)=\pi/2$.

    Now let $\sigma$ be the face of $C$ whose vertices have distance less than $\pi/2$ to $x$ and let $\tau=C\setminus\sigma$ be the complementary face. Then $x$ lies in $\overline{\sigma}$ and $y$ is a point of $\overline{\tau}$, because the distances from a point to the vertices of the simplex carrying that point are less than $\pi/2$. For $u\in\vt(\sigma)$ and $v\in\vt(\tau)$, we know that $d(x,v)=\pi/2$, $d(x,u)<\pi/2$, and $\angle_u(x,v)\leq\pi/2$. Hence, $d(v,u)=\pi/2$ by the spherical law of cosines.
\end{proof} 

%% file: cccomp.tex
%
%   cccomp.tex
%
%   23.06.2010
%

\section{Coconvex supported subcomplexes}

\begin{definition}
    Let $\Lambda$ be a simplicial complex and let $M$ be a subset of $\Lambda$. By $\Lambda(M)$ we denote the maximal subcomplex of $\Lambda$ contained in $M$. We shorten $\Lambda'(M)=\Lambda'(M\cap\Lambda')$ for a subcomplex $\Lambda'\subseteq\Lambda$. The set $M$ is said to be a support of $\Lambda(M)$. A subcomplex $\Lambda'\subseteq\Lambda$ is a (P) supported subcomplex (of $\Lambda$) if and only if $\Lambda'$ admits a support with property (P).
\end{definition}

\noindent The objects of our observation are the connectedness properties of coconvex supported subcomplexes of spherical buildings. This section is dedicated to the proof of the first main result.

\begin{maintheorem}\label{theorem-a}
    Non empty, closed, coconvex supported subcomplexes of spherical buildings are homotopy Cohen--Macaulay. They are non contractible for thick buildings of dimension at least one.
\end{maintheorem}

\noindent Note that coconvex supported subcomplexes are not coconvex in general. But for a coconvex set $M$ and a simplex $\sigma$ which is not contained in neither $M$ nor its complement, $M\cap\partial\sigma$ is a strong deformation retract of $M\cap\overline{\sigma}$. We can therefore construct a sequence $\Lambda=\Lambda_n\supseteq\cdots\supseteq\Lambda_0=\Lambda(M)$ of subcomplexes such that the maximal dimension of simplices like $\sigma$ is decreasing and $\Lambda_i \cap M$ is a strong deformation retract of $\Lambda_{i+1}\cap M$. This means that  coconvex supported subcomplexes are homotopy equivalent to their coconvex supports. We record this observation.
\begin{figure}[ht]
    \begin{center}
        \includegraphics[scale=0.8]{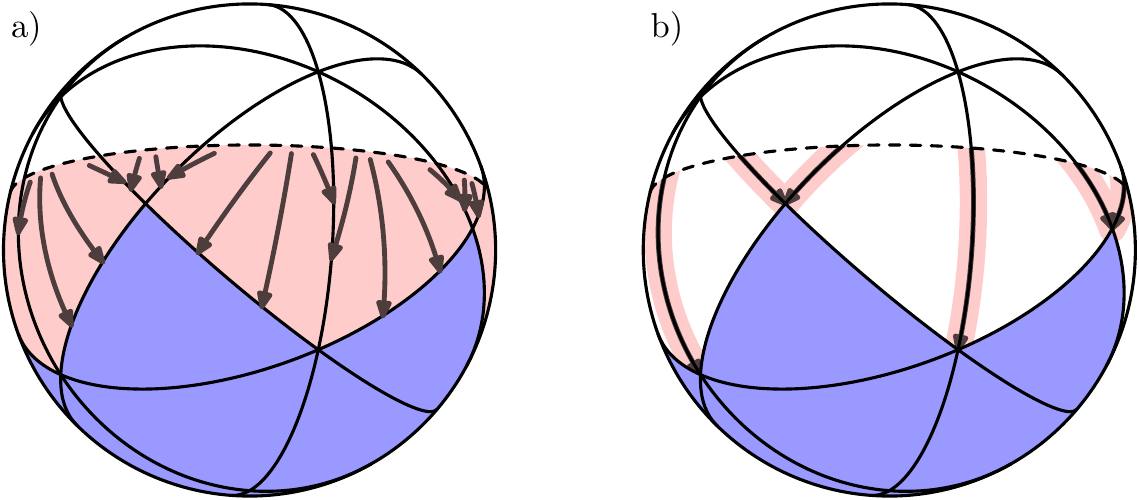}
    \end{center}
    \caption{Deformation of a coconvex support onto the supported subcomplex.}
\end{figure}

\begin{observation}\label{ccc:supp-comp-homotopy}
    Let $\Lambda$ be a subcomplex and let $M$ be a coconvex subset of $\Delta$. Then $\Lambda\cap M$ and $\Lambda(M)$ are homotopy equivalent.
\end{observation}

\begin{lemma}\label{ccc:openconvexsets}
     Suppose $\dim\Delta>0$. If $M$ is an open, convex subset of $\Delta$ containing a pair of antipodes, then $M=\Delta$.
\end{lemma}
\begin{proof}
    Let $x,y\in M$ be antipodal points, contained in some apartment $\Sigma$. Then $\Sigma$ is contained in $M$, since the convex hull of $x$ together with a neighborhood of $y$ covers $\Sigma$. The closure of any chamber $C$ that intersects $M$ is contained in $M$, since $C$ has an opposite chamber in $\Sigma$ and we therefore get an apartment in $M$ (as above) containing both. Now the assertion follows by induction on the gallery distance from $\Sigma$.
\end{proof}

\begin{notation}
    If  $\sim$ is one of the relations $<$, $\leq$, $>$, $\geq$ or $=$ and $x\in\Delta$ we put $\Omega^\sim_\Delta(x) = \{y\in\Delta\mid d(x,y)\sim\pi/2\}$.
\end{notation}

\begin{prop}\label{ccc:sphericity}
    Closed, coconvex supported subcomplexes of $\Delta$ are $\dim\Delta$-spherical or empty.
\end{prop}
\begin{proof}
    Let $M$ be a non empty, closed, and coconvex subset of $\Delta$. Since the case $\dim\Delta=0$ is trivial and the case $\Delta=\Delta(\Delta)$ is covered by the Solomon--Tits theorem, suppose that $M$ is a proper subset and $\dim\Delta>0$.

    Let $\Sigma$ be an apartment whose intersection with $M$ is not empty. Then $\Sigma\setminus M$ is contained in an open hemisphere of $\Sigma$, since it is a proper, open, convex subset. Therefore $\Sigma\cap M$ contains a closed chamber by \ref{building:chamberdiameter}. Hence, $\Delta(M)$ is $\dim\Delta$--dimensional.

    Let $C$ be a chamber not contained in $M$. Any finite subcomplex of $\Delta(M)$ is coverable by a finite set $\{\Sigma_1,\ldots,\Sigma_m\}$ of apartments, each of which contains $C$. Let us denote $\Lambda_r=\Sigma_1\cup\cdots\cup\Sigma_r$ and $\Psi_r=\Sigma_r\cap\Lambda_{r-1}$. According to \cite[Lemma~3.5]{vh} we may choose the set of apartments and their order such that, for any $r$, $\Psi_r$ is an union of roots in $\Sigma_r$. We prove the ($\dim\Delta-1$)--connectedness of $\Lambda_m(M)$. Clearly, as $\Sigma_r\cap M$ is closed and coconvex, $\Sigma_r(M)\approx\Sigma_r\cap M$ is ($\dim\Delta-1$)--connected. Hence, the desired assertion follows from \ref{mainconstr}~a), once we show that $\Psi_r\cap M\approx\Psi_r(M)=\Sigma_r(M)\cap\Lambda_{r-1}(M)$ is ($\dim\Delta-2$)--connected.

    Let $x$ be a point of $C\setminus M$ and let $p:\Sigma_r\setminus\opp^*(x)\maps\partial(\Sigma_r\cap M)$ denote the geodesic projection with center $x$ onto the boundary of $\Sigma_r\cap M$. Since $\Psi_r$ is an union of roots, each of which contains $x$ as an inner point, we know that $\Psi_r$ is star shaped with respect to $x$ and does not contain the antipode of $x$ in $\Sigma_r$. Therefore the restriction of $p$ to $\Psi_r\cap M$ is a retraction $\Psi_r\cap M\rightarrow\Psi_r\cap\partial(M\cap\Sigma_r)$ inducing a strong deformation retraction
    \begin{equation*}
        (\Psi_r\cap M)\times[0,1]\longrightarrow\Psi_r\cap M;\thickspace
        (z,t)\longmapsto r_\Delta(z,p(z),t)
    \end{equation*}
    from $\Psi_r\cap M$ onto $\Psi_r\cap\partial(M\cap\Sigma_r)$. Hence, $\Psi_r\cap M\approx\Psi_r\cap\partial(M\cap\Sigma_r)$.
    \begin{figure}[ht]
        \begin{center}
            \includegraphics[scale=0.8]{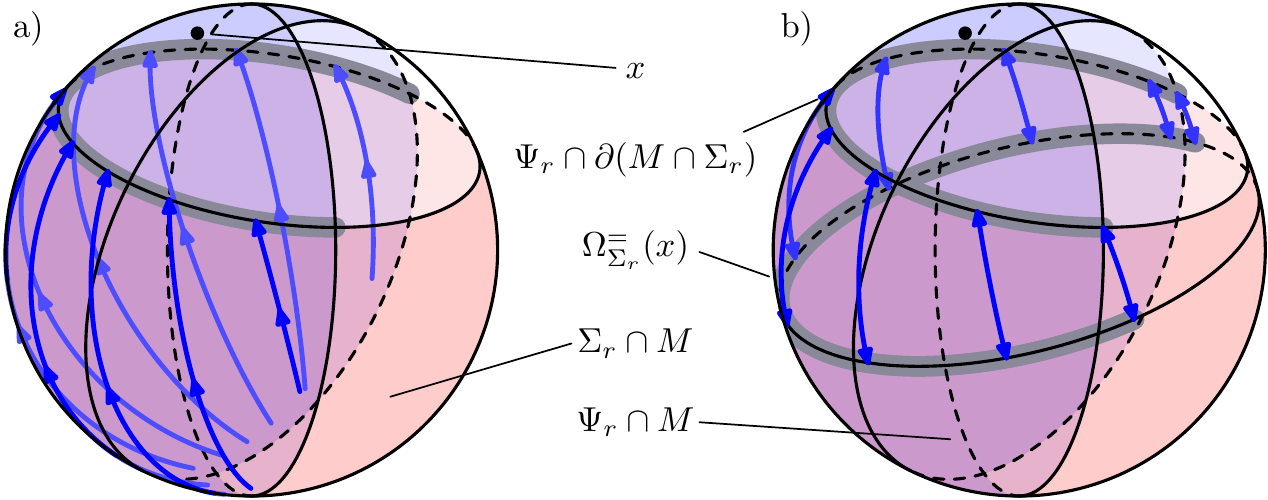}
        \end{center}
        \caption{Via geodesic projection with center $x$ we get a) a homotopy equivalence $\Psi_r\cap M\approx\Psi_r\cap\partial(M\cap\Sigma_r)$ and b) a homeomorphism of $\Psi_r\cap\partial(M\cap\Sigma_r)$ and a coconvex subset of $\Omega_{\Sigma_r}^=(x)$.}
    \end{figure}
    Observe that $p$ maps $(\Sigma_r\setminus M)\setminus\Psi_r$ onto the complement $\partial(M\cap\Sigma_r)\setminus\Psi_r$ of the above retract. Let $q:\Sigma_r\setminus\opp^*(x)\rightarrow \Omega_{\Sigma_r}^=(x)$ denote the geodesic projection with center $x$ onto the equator $\Omega_{\Sigma_r}^=(x)$. Note that the restriction of $q$ to $\partial(M\cap\Sigma_r)$ is the inverse homeomorphism of $p|_{\Omega_{\Sigma_r}^=(x)}$ and that $q=q\circ p$. Furthermore $q$ maps open, convex subsets of $\Sigma_r\setminus\opp^*(x)$ to open, convex sets. Since $(\Sigma_r\setminus M)\setminus\Psi_r$ is open and convex, $q(\partial(M\cap\Sigma_r)\setminus\Psi_r)=q((\Sigma_r\setminus M)\setminus\Psi_r)$ is an open, convex subset of $\Omega_{\Sigma_r}^=(x)$. Therefore $\Psi_r\cap\partial(M\cap\Sigma_r)$ is ($\dim\Delta-2$)--connected, since it is the homeomorphic image of a closed, coconvex subset of $\Omega_{\Sigma_r}^=(x)$.
\end{proof}

\begin{lemma}\label{ccc:proj-openconv-on-link}
     Suppose $\dim\Delta>0$. Let $x, y\in\Delta$ be opposite vertices and let $M$ be a proper, open, convex subset of $(x,y)$ (see \ref{building:segments}). Then the image of $M$ under the geodesic projection onto $\lk x$ is a proper, open, convex subset of $\lk x$.
\end{lemma}
\begin{proof}
    Let $A$ be an arbitrary apartment of $\lk x$. The convex hull of $x, y,$ and $A$ is an apartment $\Sigma$. For a point $z\in A$, the geodesic segment joining $x$ and $y$ going through $z$ is contained in $\Sigma$. Therefore $\Sigma$ contains the preimage of $z$ under the restriction $p_x|_{(x,y)}$. Hence, $p_x(M)\cap A=p_x(M\cap\Sigma)$.

    Let $q$ be the geodesic projection with center $x$ onto the equator $\Omega_{\Sigma}^=(x)$. From $q=q\circ p_x|_\Sigma$ we get $q(M\cap\Sigma)=q(p_x(M)\cap A)$. Therefore $q(p_x(M)\cap A)$ is an open, convex subset of $\Omega_{\Sigma}^=(x)$, because $q$ maps open, convex subsets of $\Sigma\setminus\opp^*(x)$ to open, convex subsets of $\Omega_{\Sigma}^=(x)$. Since the restriction of $q$ on $A$ is an isometry according to \ref{building:link-metric}, $p_x(M)\cap A$ is open and convex in $\lk x$. Recall that $A$ has been chosen arbitrary. Hence, $p_x(M)$ is open and convex in $\lk x$.

    Assume there are $u,v\in M$ such that\ $\angle_x(p_xu,p_xv)=\pi/2$. Then $[x,u]\cup[u,y]\cup[y,v]\cup[v,x]$ would be a great circle. But this is impossible, because $M$ would contain $x$ or $y$ or a pair of antipodal points from $\Delta$. Hence, $p_x(M)$ is a proper subset of $\lk x$.
\end{proof}

\begin{cor}\label{ccc:link}
    The links in non empty, closed, coconvex supported subcomplexes of $\Delta$ are non empty, closed, coconvex supported subcomplexes.
\end{cor}
\begin{proof}
    It is sufficient to prove the assertion for vertices. Let $M\subset\Delta$ be a proper, open, convex subset and let $x$ be a vertex in $\Delta\setminus M$. Since $M$ is open, a simplex of $\lk x$ is contained in the link $\lk_{\Delta(\Delta\setminus M)}x$ if and only if its closure does not intersect the image $p_x(M\cap\st x)$. Let $y\op x$ be a vertex. Since $[x,y]$ contains a neighborhood of $x$, by \cite[Lemma~3.6.1]{kl}, it also contains $\st x$. Hence, the assertion follows from \ref{ccc:proj-openconv-on-link}.
\end{proof}

\begin{prop}\label{ccc:noncontractible}
    Non empty, closed, coconvex subsets of at least one
    dimensional, thick, spherical buildings are non contractible.
\end{prop}
\begin{proof}
    Let $\Delta$ be at least one dimensional and thick and let $M\subset\Delta$ be a proper, open, convex subset. We proof the existence of a $\dim\Delta$--dimensional sphere in $\Delta\setminus M$ by induction on $\dim\Delta$.

    Since $\Delta$ is thick, there are three pairwise opposite chambers. Then there is a pair $x\op y$ of opposite vertices inside $\Delta\setminus M$ by \ref{ccc:openconvexsets}.

    Let $S$ be the union of the open geodesic segments that join $x$, $y$ and contain a point from $M'=p_x(M\cap[x,y])$. Then $[x,y]\setminus S$ is a subset of $\Delta\setminus M$. Furthermore $[x,y]\setminus S$ is the spherical join of $\{x,y\}$ and $(\lk x\setminus M',\angle_x)$ by \cite[Proposition~3.10.1]{kl}. Hence we are done, if $\lk x\setminus M'$ contains a $(\dim\lk x)$--dimensional sphere. Clearly, that is assured by the induction hypothesis, provided that $\dim\Delta>1$, since $M'$ is a proper, open, convex subset of $\lk x$ by \ref{ccc:proj-openconv-on-link}. But even if $\dim\Delta=1$ we get a $0$--sphere in $\lk x\setminus M'$, since $M'$ is connected and $\lk x$ is thick.
\end{proof}

%% file: hcomp.tex
%
%   hcomp.tex
%
%   23.06.2010
%

\section{Hemisphere complexes}

In this section we will examine some special coconvex supported subcomplexes. Their supports are unions of hemispheres, so the complexes will be called hemisphere complexes. Throughout this section let $x$ be an arbitrary point of $\Delta$.

\begin{definition}\label{hc:definition}
    The subcomplex $\Delta^>(x)=\Delta(\Omega^>_\Delta(x))$ is said to be the open hemisphere complex of $\Delta$ with respect to the pole $x$ and $\Delta^\geq(x)=\Delta(\Omega^\geq_\Delta(x))$ is said to be the closed hemisphere complex of $\Delta$ with respect to the pole $x$. $\Delta^=(x)=\Delta(\Omega^=_\Delta(x))$ is the equator complex of $x$.
\end{definition}

\noindent By \cite[II.1~Proposition~1.4]{brha} the sets $\Omega^<_\Delta(x)$ and $\Omega^\leq_\Delta(x)$ are convex. Hence, hemisphere complexes are coconvex supported subcomplexes.

\begin{cor}\label{hc:closed-hc-main}
    Closed hemisphere complexes are homotopy Cohen--Macaulay. They are non contractible for thick buildings.
\end{cor}
\begin{proof}
    The assertion is an immediate consequence of theorem \ref{theorem-a} except for the non contractibility in the case $\dim\Delta=0$. But closed hemisphere complexes are also in this case non contractible, since $\Omega^<_\Delta(x)$ is a single point.
\end{proof}

\noindent Note that the intersection of $\Omega_\Delta^\sim(x)$ with apartments containing $x$ is always convex. We therefore get the following observation.

\begin{observation}\label{hc:full-subcomplex}
    Open hemisphere complexes, closed hemisphere complexes and equator complexes are full subcomplexes of $\Delta$.
\end{observation}

\begin{notation}
    If $\Delta$ is reducible, then $\Delta_\hor(x)$ denotes the maximal join factor of $\Delta$ that is contained in $\Omega_\Delta^=(x)$ whereas $\Delta_\ver(x)$ denotes the minimal join factor containing $x$.
\end{notation}

\noindent We certainly have $\Delta=\Delta_\hor(x)*\Delta_\ver(x)$, since any irreducible join factor that does not intersect the closure of the simplex carrying $x$ lies in $\Omega_\Delta^=(x)$ by \ref{building:join}. Now let us have a look at the join decomposition of hemisphere complexes:

Let $\Delta=\Delta_1*\Delta_2$ be a reducible spherical building and let $\sim$ be one of the relations $>$, $\geq$ or $=$. We get $\Delta^\sim(x)=\Delta_1(\Omega_\Delta^\sim(x))*\Delta_2(\Omega_\Delta^\sim(x))$ from \ref{hc:full-subcomplex}. If $x$ is a point of $\Delta_1$, then $\Delta_1(\Omega^\sim_\Delta(x))=\Delta^\sim(x)$ and $\Delta_2$ is a subcomplex of $\Delta_\hor(x)$. Therefore $\Delta_2(\Omega_\Delta^\sim(x))$ is empty if $\sim$ is a strong inequality or all of $\Delta_2$ otherwise. If $x$ is not contained in neither $\Delta_1$ nor $\Delta_2$ then there are two unique points $x_1\in\Delta_1$ and $x_2\in\Delta_2$ such that $x$ lies inside their joining segment. In this case we have $\Delta_i(\Omega^\sim_\Delta(x))=\Delta_i^\sim(x_i)$:

Assume $\{i,j\}=\{1,2\}$ and $y\in\Delta_i$. Then $d(x_1,x_2)=\pi/2=d(y,x_j)$, since points of disjoint factors have distance $\pi/2$. There is an apartment containing $x$, $x_1$, $x_2$ and $y$. We therefore get from the spherical law of cosines
\begin{align*}
    \cos d(x,y) &= \sin d(x_j,x) \cos\angle_{x_j}(x,y) \\
    &= \sin d(x_j,x)\cos\angle_{x_j}(x_i,y)
    = \sin d(x_j,x)\cos d(x_i,y).
\end{align*}
Hence, $d(x,y)\sim\pi/2$ if and only if $d(x_i,y)\sim\pi/2$. We proved:

\begin{prop}\label{hc:join}
    Assume $\Delta_\ver(x)=\Delta_1*\cdots*\Delta_k$ is a decomposition of $\Delta_\ver(x)$ into irreducible factors. Then $\Delta^>(x)=\Delta_\ver^>(x)=\Delta_1(\Omega_\Delta^>(x))*\cdots*\Delta_k(\Omega_\Delta^>(x))$ is a join of open hemisphere complexes in $\Delta_1,\ldots,\Delta_k$ and the equator complex decomposes to $\Delta^=(x)=\Delta_\ver^=(x)*\Delta_\hor(x)$.
\end{prop}

\noindent If $\Delta$ is irreducible, then the closed stars of the simplices opposite to the simplex carrying $x$ are contained in $\Omega_\Delta^>(x)$, since the diameter of closed chambers is less than $\pi/2$ by \ref{building:join}. Hence, open hemisphere complexes of irreducible spherical buildings have the same dimension as the surrounding building. In general, $\dim\Delta^>(x)=\dim\Delta_\ver(x)\leq\dim\Delta$ by \ref{hc:join}; and the last inequality is strict if $\Delta_\hor(x)$ is not empty. In the sequel we will have to take care of this case.

\begin{lemmadef}\label{hc:equator-link}
    Let $\sigma$ be a simplex of $\Delta^=(x)$. There is a point $p_\sigma x\in\lk\sigma$ such that
    \begin{equation*}
        d(x,y)\sim\pi/2\Longleftrightarrow d_{\lk\sigma}(p_\sigma x,y)\sim\pi/2,
    \end{equation*}
    for any point $y\in\lk\sigma$ and any relation $<$, $\leq$, $=$, $\geq$ or $>$. For the simplex $\xi$ carrying $x$ and the simplex $\chi$ carrying $p_\sigma x$, holds $\sigma\cup \chi=\proj_\sigma\xi$. (If $\sigma$ is a vertex then $p_\sigma x$ is the geodesic projection of $x$ on $\partial\st\sigma=\lk\sigma$.)
\end{lemmadef}
\begin{proof}
    At first we will prove the assertion for vertices. Let  $z\in\Delta^=(x)$ be a vertex and let $y\in\lk z$ be a point. There is an apartment containing $x$, $y$, and $z$. By the spherical law of cosines and \ref{building:link-metric} we therefore get
    \begin{equation*}
        \cos d(x,y) = \sin d(z,y)\cos\angle_z(x,y)
                    = \sin d(z,y)\cos d_{\lk z}(p_zx,y).
    \end{equation*}
    Hence, $d(x,y)\sim\pi/2$ if and only if $d_{\lk z}(p_zx,y)\sim\pi/2$. The assertion on the projection is an immediate consequence of \ref{building:projection}.

    One inductively obtains the general assertion and the definition of $p$ by regarding a simplex as a vertex in the link of one of its codimension-1-faces. This is justified, since $\proj_\tau\xi=\proj_\tau\proj_\sigma\xi$ for $\sigma\leq\tau$ by \cite[2.30.5]{ti1} and since the canonical metric is unique.
\end{proof}

\begin{lemma}\label{hc:link}
    The links in open hemisphere complexes of irreducible buildings are non empty, closed, coconvex supported subcomplexes.
\end{lemma}
\begin{proof}
    Suppose $\Delta$ is irreducible. Let $\sigma$ be a simplex of $\Delta^>(x)$. The idea is to recognize $\lk_{\Delta^>(x)}\sigma$ as a link $\lk_{\Delta^\geq(x')}\sigma$ in a closed hemisphere complex (to a slightly perturbed pole $x'$) and to apply \ref{ccc:link}. The task is to choose $x'$.

    Let $y$ be a point of $\sigma$. As $D=\{\abs{d(x,z)-\pi/2}\mid z\in\vt(\Delta)\setminus\vt(\Delta^=(x))\}$ is finite by \ref{building:retraction}, we may chose a point $x'$ on a segment joining $x$ and $y$ such that $0<d(x,x')<\min D$. From the triangle inequality we get the implications $d(x,z)>\pi/2\Rightarrow d(x',z)>\pi/2$ and $d(x,z)<\pi/2\Rightarrow d(x',z)<\pi/2$ for any vertex $z$ of $\Delta$, hence
    \begin{equation*}
        \vt(\Delta^>(x))\subseteq\vt(\Delta^>(x'))\subseteq
        \vt(\Delta^\geq(x'))\subseteq\vt(\Delta^\geq(x)).
    \end{equation*}
    Therefore $\lk_{\Delta^>(x)}\sigma$ is contained in $\lk_{\Delta^\geq(x')}\sigma$. Let $z$ be a vertex of $\lk_{\Delta^\geq(x')}\sigma$.  According to \ref{building:join} $d(y,z)$ is less than $\pi/2$. If $y$ is an antipode of $x$, then we have $d(x,z)>\pi/2$. If $x$ and $y$ are not antipodal we deduce $d(x,z)>d(x',z)$ from $d(x,y)>d(x',y)>\pi/2$, $\angle_y(x,z)=\angle_y(x',z)$ and the spherical law of cosines. Hence, $\lk_{\Delta^>(x)}\sigma=\lk_{\Delta^\geq(x')}\sigma$, because their vertex sets coincide. Now the lemma follows from \ref{ccc:link}.
 \end{proof}

\noindent For the remainder of this section our task will be to show that open hemisphere complexes of thick, spherical buildings are spherical and non contractible. (To be more precise: $\Delta^>(x)$ is $\dim\Delta_\ver(x)$--spherical.) We already know by \ref{hc:join}, \ref{hc:full-subcomplex}, and \ref{hc:link} that links in open hemisphere complexes are joins of open hemisphere complexes and closed, coconvex supported subcomplexes. Hence, from theorem~\ref{theorem-a} we will get the second main result.

\begin{maintheorem}\label{theorem-b}
    Open hemisphere complexes of thick, spherical buildings are non contractible homotopy Cohen--Macaulay complexes.
\end{maintheorem}

\noindent Note that the proof of \ref{ccc:sphericity} would not work in the case of open hemisphere complexes, since the intersection of an open hemisphere with an union of closed hemispheres is not ($\dim\Delta-2$)-connected in general. A closer look at the proof of \cite[Lemma~3.5]{vh} suggests that such situations would be inevitable. For classical buildings one is able to achieve a precise description of the links in hemisphere complexes. Therefore I tried to mimic the sphericity proofs H.~Abels and P.~Abramenko used in \cite{al2} and \cite{ab4}, but I was not able to avoid limitations on the thickness of the underlying buildings. This led to a different approach: Starting with a closed hemisphere complex, which is known to be spherical by \ref{hc:closed-hc-main}, we delete the stars of simplices contained in the equator complex by a filtration such that the boundary of the deleted stars is contractible in the remaining subcomplex (see figure \ref{figure:contractingstars}). To do this, we will have to spend some work in advance.

\begin{figure}[ht]
\begin{center}
    \includegraphics[scale=0.8]{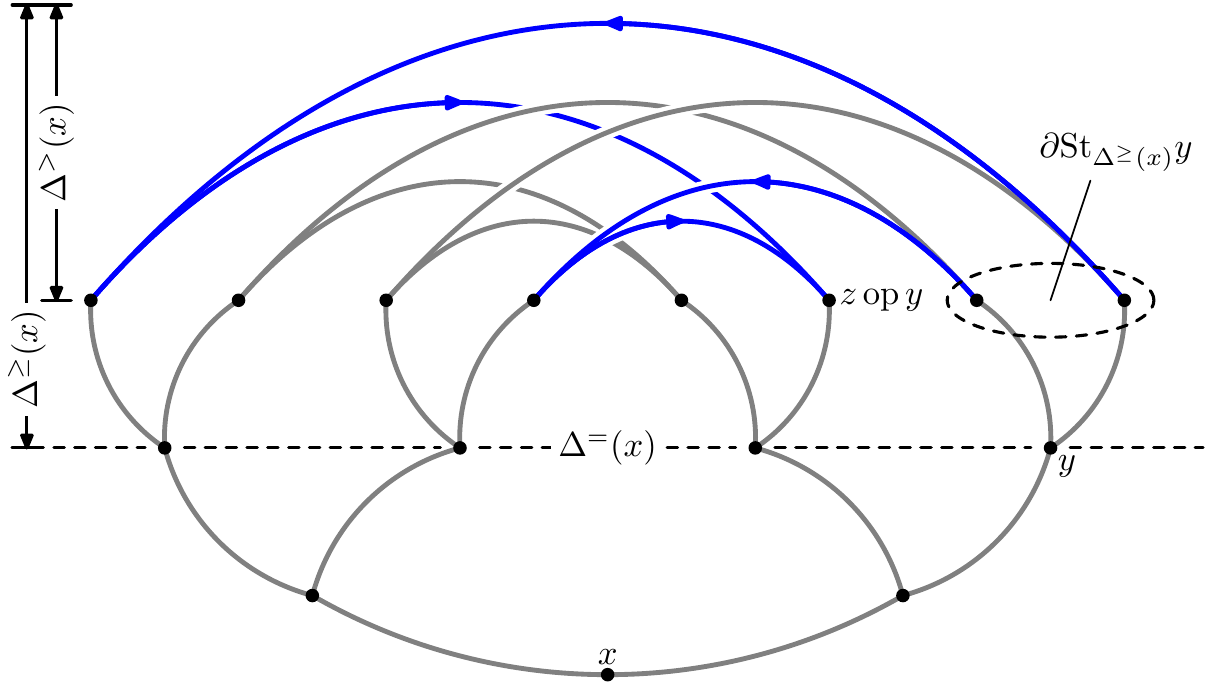}
\end{center}
    \caption{\label{figure:contractingstars}The segments joining $\partial\st_{\Delta^\geq(x)}y$ with $z\op y$ are contained in $\Delta^>(x)$. Hence, the boundary of the $\st_{\Delta^\geq(x)}y$ is contractible in $\Delta^>(x)$. ($\Delta$ is the flag complex of the projective plane over $\finitefield{2}$.)}
\end{figure} 

%% file: filter.tex
%
%   filter.tex
%
%   19.06.2010
%

\subsection{A filtration of closed hemisphere complexes}

Our aim is to define a filtration $\Delta^>(x)=F_0\subset F_1\subset\cdots\subset F_N$ of $F_N=\Delta^\geq(x)$ such that a) $F_k$ is the disjoint union $F_{k-1}\cup\bigcup_{\sigma\in I_k}\st_{F_k}\sigma$ for some set of simplices $I_k\subseteq\simp(F_k)\setminus\simp(F_{k-1})$ and b) the boundary of the relative star $\st_{F_k}\sigma$ is contractible in $F_{k-1}$, for any $\sigma\in I_k$. We begin by describing the major obstacle that needs to be overcome:

For a point $y$ of the equator complex, suppose there is an antipode $z\in\Delta^>(x)$. (In the next subsection we will show that such an antipode exists, provided $\Delta$ is thick.) Any point $u\in\partial\st_{\Delta^\geq(x)}y$ is connected to $z$ by a segment. In an ideal world, we could therefore contract $\partial\st_{\Delta^\geq(x)}\sigma$ inside $\Omega^>(x)$ by geodesically coning off from $z$. This idea works sometimes, but if $u$ lies also in the equator, we would like to see $(u,z]\subseteq\Omega_\Delta^>(x)$. This, however, does not always happen. We explain the obstruction.

Let $\sigma$ be the simplex carrying $x$ and let $\theta$ denote the simplex carrying $u$. As $(u,z]$ is contained in a geodesic segment joining $y$ and $z$, the initial segment $(u,z]\cap\st\theta$ of $(u,z]$ lies in a simplex $\tau$ of $\st\theta$ that is opposite to $\sigma\cup\theta$ in $\st\theta$. If $\sigma\setminus\theta$ is a simplex of a join factor of $\lk\theta$ that lies in $\Omega_\Delta^=(x)$, then $\tau$ is contained in the equator, since $\tau\setminus\theta$ is also a simplex of that join factor. In this case $(u,z]$ can not be contained in $\Omega_\Delta^>(x)$ regardless of which antipode $z$ of $y$ we use.

To avoid this problem, we construct a filtration such that the links of the simplices in the boundaries of the relative stars do not have join factors in $\Omega_\Delta^=(x)$. This will be done by defining a restriction map $\simp(\Delta^\geq(x))\maps\simp(\Delta^=(x))$,
mapping each simplex to one of its faces and decomposing $\Delta^\geq(x)$ into the relative stars we would like to have. Afterwards the filtration is obtained by a partial ordering on the image of the restriction map.

\begin{definition}\label{filter:def-restriction}
    For a chamber $C\in\cham(\Delta)$ and a vertex $v$ of $C$ let $C_v=C\setminus v$ denote its complementary face. For a simplex $\sigma\in\simp(\Delta)$ let $\sigma_x^=$ be its maximal face contained in the equator complex with respect to the pole $x$. The map $\restr_\Delta^x:\simp(\Delta^\geq(x))\longrightarrow\simp(\Delta^=(x))$ defined by
    \begin{equation*}
        \vt\restr_\Delta^x(\sigma)=\{ v\in\vt(\sigma_x^=)\mid \exists\
        C\in\cham(\st\sigma_x^=)\thickspace :\thickspace \lk
        C_v\not\subseteq\Delta^=(x)\}
    \end{equation*}
    (see figure \ref{figure:restriction-map}) is called the ($\Delta,x$)--restriction on $\Delta^\geq(x)$.
\end{definition}

\begin{figure}[ht]
\begin{center}
    \includegraphics[scale=0.8]{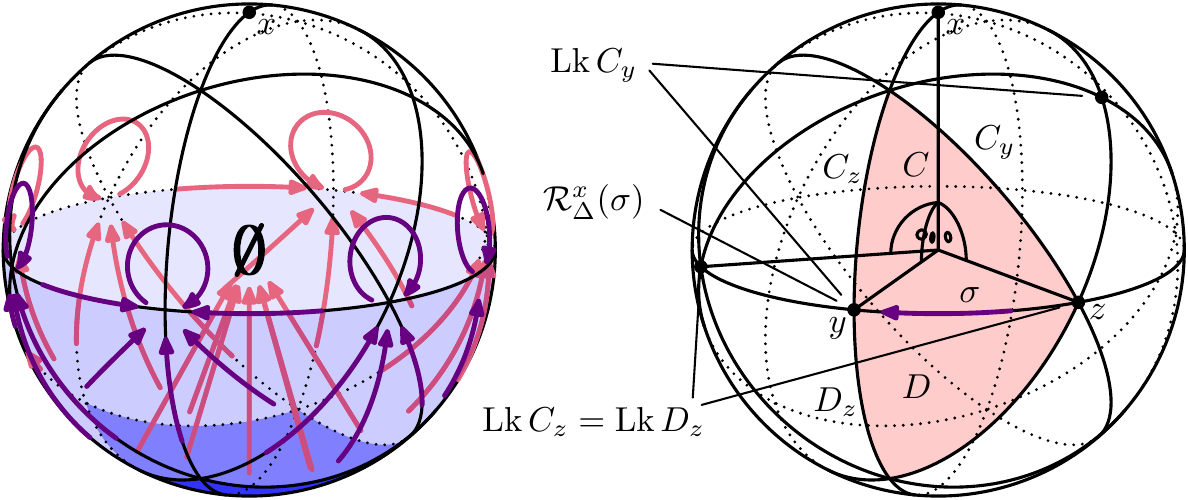}
\end{center}
    \caption{\label{figure:restriction-map}The ($\Delta,x$)-restriction in an apartment of $\mathrm{Flag}(\projspace{3}(F))$ containing $x$. ($x$ is the midpoint of an edge joining a point and a hyperplane of $\projspace{3}(F)$.)}
\end{figure}

\begin{lemma}\label{filter:join}
    For $x, y\in\Delta$ the following statements are equivalent:
    \begin{itemize}
    \item[\rm{a)}]
        $\Delta=\Delta_x*\Delta_y$ decomposes as a spherical join, s.t. $x\in\Delta_x$ and $y\in\Delta_y$.
    \item[\rm{b)}]
        $\lk C_v\subseteq\Delta^=(x)$ for any chamber $C$ of $\st y$ and any vertex $v$ of the simplex carrying $y$.
    \end{itemize}
\end{lemma}
\begin{proof}
    The implication a) $\Rightarrow$ b) is clear. Suppose b) holds. Let $\Sigma$ be an apartment containing $x$ and $y$. Then b) means that $x$ is contained in any wall bordering $\st_\Sigma y$. Therefore $x$ is a point of $\overline{\st_\Sigma y}$. Since $d(x,v)=\pi/2$ for any vertex $v$ of the simplex carrying $y$, we get $d(x,y)=\pi/2$ from \ref{building:chamberdiameter}, \ref{building:link-metric} and the spherical law of cosines. Hence, a) follows by \ref{building:join}.
\end{proof}

\begin{cor}\label{filter:empty}
    Let $\sigma$ be a simplex of $\Delta^\geq(x)$. Then
    \begin{equation*}
        \restr_\Delta^x(\sigma)=\emptyset \Longleftrightarrow
        \sigma^=_x \text{ is a simplex of }\Delta_\hor(x).
    \end{equation*}
\end{cor}

\begin{notation}
    For any simplex $\sigma\in\simp(\Delta)$ let $\lambda_\sigma$ denote the simplicial map $\lambda_\sigma:\overline{\st\sigma}\rightarrow\st\sigma; \tau\mapsto\sigma\cup\tau$.
\end{notation}

\begin{lemma}\label{filter:link}
    Let $\tau$ be a simplex of $\Delta^=(x)$ and let $\sigma$ be a face of $\tau$. Then we have $\restr_{\lk\sigma}^{p_\sigma x}(\tau\setminus\sigma)=\restr_\Delta^x(\tau)\setminus\sigma$ and $\restr_\Delta^x(\tau)\cap\sigma$ is a face of $\restr_\Delta^x(\sigma)$.
\end{lemma}
\begin{proof}
    $\lambda_\sigma$ identifies $\st_{\lk\sigma}\tau\setminus\sigma$ with $\st\tau$. Note that for any vertex $v\in\vt(\tau\setminus\sigma)$ and any chamber $C$ of
    $\st_{\lk\sigma}(\tau\setminus\sigma)$ we have $\lk_{\lk\sigma} C_v=\lk (\lambda_\sigma C)_v$. By \ref{hc:equator-link} we know that $(\lk\sigma)^=(p_\sigma x)$ equals $\lk\sigma\cap\Delta^=(x)$. Hence, using the definition one gets $\restr_{\lk\sigma}^{p_\sigma x}(\tau\setminus\sigma)=\restr_\Delta^x(\tau)\setminus\sigma$. The second assertion is clear, since $\st\sigma$ contains $\st\tau$.
\end{proof}

\begin{cor}\label{filter:idempotent}
    The map $\restr_\Delta^x$ is idempotent and its image
    $\im\restr_\Delta^x$ is a subcomplex of $\Delta_\ver(x)$.
\end{cor}
\begin{proof}
    By \ref{filter:link} any face $\sigma$ of $\restr_\Delta^x(\tau)$ is a face of $\restr_\Delta^x(\sigma)$, hence $\restr_\Delta^x(\sigma)=\sigma$ for all faces of $\restr_\Delta^x(\tau)$. According to \ref{filter:empty} a simplex $\sigma\in\im\restr_\Delta^x$ is contained $\Delta_\hor(x)$ if and only if $\sigma=\restr_\Delta^x(\sigma)=\emptyset$.
\end{proof}

\begin{lemma}\label{filter:equal-restriction}
    Let $\tau$ be a simplex of $\Delta^\geq(x)$ and let $\sigma$ be a face of $\tau^=_x$. The following statements are equivalent:
    \begin{itemize}
        \item[{\rm a)}]
            $(\tau\setminus\sigma)^=_x$ is a simplex in $(\lk\sigma)_\hor(p_\sigma x)$.
        \item[{\rm b)}]
            $\restr_\Delta^x(\tau)$ is a face of $\sigma$.
        \item[{\rm c)}]
            $\restr_\Delta^x(\tau)=\restr_\Delta^x(\sigma)$.
    \end{itemize}
\end{lemma}
\begin{proof}
    We may suppose $\tau=\tau^=_x$. By \ref{filter:link} and \ref{filter:empty} the simplex $\tau\setminus\sigma$ is contained in $(\lk\sigma)_\hor(p_\sigma x)$ if and only if $\restr_\Delta^x(\tau)\setminus\sigma=\emptyset$ holds, within if and only if $\restr_\Delta^x(\tau)$ is a face of $\sigma$. Hence, a) $\Leftrightarrow$ b).

    The implication c) $\Rightarrow$ b) is obvious, since $\restr_\Delta^x(\sigma)$ is always a face of $\sigma$. We show b) $\Rightarrow$ c): Assume $\restr_\Delta^x(\tau)$ is a face of $\sigma$. Then $\restr_\Delta^x(\tau)$ is a face of $\restr_\Delta^x(\sigma)$ according to \ref{filter:link}. Conversely, $\tau\setminus\restr_\Delta^x(\tau)$, within its face $\restr_\Delta^x(\sigma)\setminus\restr_\Delta^x(\tau)$ is contained in $(\lk\restr_\Delta^x(\tau))_\hor(p_{\restr_\Delta^x(\tau)}x)$ by b) $\Rightarrow$ a). Then $\restr_\Delta^x(\sigma)$ is a face of $\restr_\Delta^x(\tau)$ by \ref{filter:link} and \ref{filter:idempotent}.
\end{proof}

\noindent By lemma \ref{filter:equal-restriction} the $\restr_\Delta^x$--preimage of $\tau\in\simp(\im\restr_\Delta^x$) is $\st_F\tau$ for some subcomplex $F\subseteq\Delta^\geq(x)$. Moreover it holds $\lk_F\tau=(\lk\tau)^>(p_\tau x)*(\lk\tau)_\hor(p_\tau x)$ as desired. It also would have been possible to use the equivalence of lemma \ref{filter:equal-restriction} as a definition of the restriction map. This was done in \cite{buwo2} and \cite{bgw} by K.-U.~Bux, R.~Gramlich, S.~Witzel and K.~Wortman to define a filtration of Euclidean buildings whose relative links should get the above properties.

\begin{definition}\label{filter:def-order}
    For two simplices $\sigma$, $\tau$ of $\im\restr_\Delta^x$ we define $\sigma\preceq\tau$ if and only if the upper bound $\sigma\cup\tau$ exists and $\restr_\Delta^x(\sigma\cup\tau)=\tau$. We write $\sigma\prec\tau$ if and only if $\sigma\preceq\tau$ but not $\tau\preceq\sigma$.
\end{definition}

\begin{lemma}\label{filter:order}
    Let $\sigma$ be a simplex of $\Delta^=(x)$ and let $\tau$ be a simplex of $\im\restr_\Delta^x$. Suppose $\restr_\Delta^x(\sigma)\preceq\tau$. Then $\sigma\cup\tau$ exists and $\restr_\Delta^x(\sigma\cup\tau)=\tau$.
\end{lemma}
\begin{proof}
    We put $\theta=\restr_\Delta^x(\sigma)$. According to \ref{filter:equal-restriction} $\sigma\setminus\theta$ is a simplex of $(\lk\theta)_\hor(p_{\theta}x)$, whereas $\tau\setminus\theta=\restr_\Delta^x(\tau\cup\theta)\setminus\theta=\restr_{\lk\theta}^{p_{\theta}x}(\tau\setminus\theta)$ is contained in $(\lk\theta)_\ver(p_{\theta}x)$ by \ref{filter:idempotent}. Hence, $\sigma\cup\tau$ exists.

    $\restr_\Delta^x(\sigma\cup\tau)\cap\sigma$ is a face of $\theta$ by \ref{filter:link}. Since $\sigma\setminus\theta$ and
    $\restr_\Delta^x(\sigma\cup\tau)$ are disjoint, the later is a face of $\theta\cup\tau$. From \ref{filter:equal-restriction} we get $\restr_\Delta^x(\sigma\cup\tau)=\restr_\Delta^x(\theta\cup\tau)=\tau$.
\end{proof}

\begin{cor}\label{filter:poset}
    $(\im\restr_\Delta^x,\preceq)$ is a poset with unique minimal
    element $\emptyset$.
\end{cor}
\begin{proof}
    Let $\sigma$, $\tau$, and $\theta$ be simplices of $\im\restr_\Delta^x$. Suppose $\sigma\preceq\tau$ and $\tau\preceq\theta$. Then the upper bound $\sigma\cup\tau\cup\theta$ exists and $\restr_\Delta^x(\sigma\cup\tau\cup\theta)=\theta$ by \ref{filter:order}. From \ref{filter:equal-restriction} follows $\restr_\Delta^x(\sigma\cup\theta)=\theta$, since $\restr_\Delta^x(\sigma\cup\tau\cup\theta)$ is a face of $\sigma\cup\theta$. Hence, $\sigma\preceq\tau$.
\end{proof}

\noindent Suppose $\sigma_0\prec\sigma_1\prec\cdots$ is a strictly increasing sequence in $(\im\restr_\Delta^x,\preceq$). According to \ref{filter:order} the upper bounds $\sigma_0\cup\sigma_1\cup\cdots\cup\sigma_k$ exist and the sequence of their ranks is strictly increasing. Hence, the length of any strictly increasing sequence in $(\im\restr_\Delta^x,\preceq$) is bounded from above by $\rank\Delta^=_\ver(x)$. I.e. strictly increasing sequences are finite.

\begin{definition}\label{filter:def-height}
    For a simplex $\sigma$ of $\Delta^\geq(x)$ let its height $\height(\sigma)$ be the length of the longest strictly increasing chain in $(\im\restr_\Delta^x,\preceq)$ ending with $\restr_\Delta^x(\sigma)$.
\end{definition}

\begin{lemma}\label{filter:faceheight}
    Let $\tau$ be a simplex of $\Delta^\geq(x)$. Then $\height(\sigma)\leq\height(\tau)$, for any face $\sigma$ of $\tau$. Equality holds if and only if $\restr_\Delta^x(\sigma)=\restr_\Delta^x(\tau)$.
\end{lemma}
\begin{proof}
    According to \ref{filter:equal-restriction} we have
    $\restr_\Delta^x(\restr_\Delta^x(\sigma)\cup\restr_\Delta^x(\tau))=
    \restr_\Delta^x(\tau)$, for any face $\sigma$ of $\tau$ (replace $\sigma$ by $\restr_\Delta^x(\sigma)\cup\restr_\Delta^x(\tau)$ and use b) $\Rightarrow$ c)), hence $\restr_\Delta^x(\sigma)\preceq\restr_\Delta^x(\tau)$ and consequently $\height(\sigma)\leq\height(\tau)$. By definition, we have either $\restr_\Delta^x(\sigma)=\restr_\Delta^x(\tau)$ or $\restr_\Delta^x(\sigma)\prec\restr_\Delta^x(\tau)$ (which means $\height(\sigma)<\height(\tau)$).
\end{proof}

\begin{figure}[ht]
\begin{center}
    \includegraphics[scale=0.8]{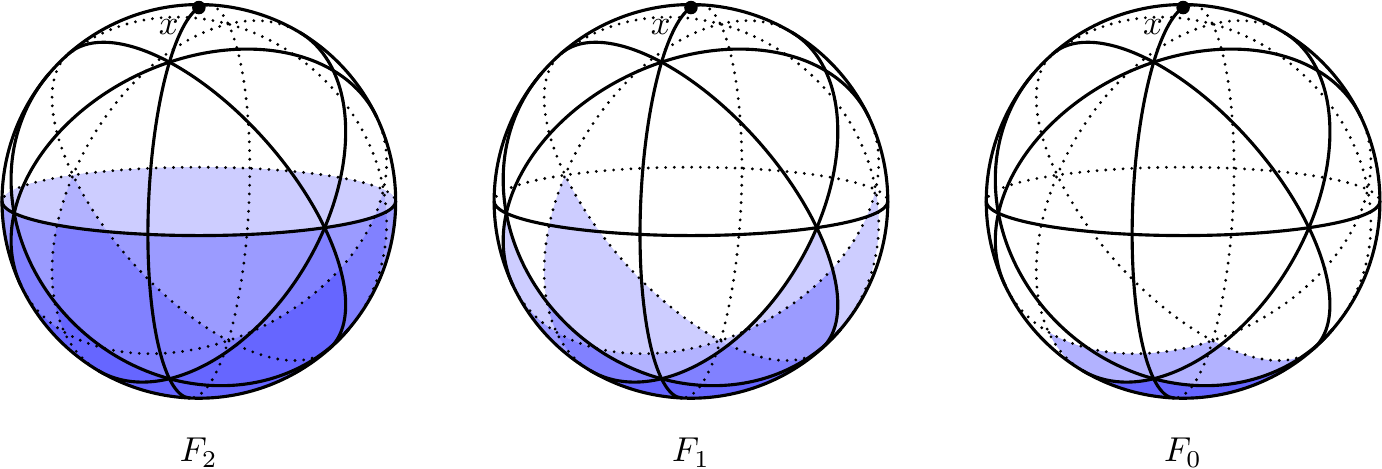}
\end{center}
    \caption{\label{figure:the-filtration}The filtration of an apartment of $\mathrm{Flag}(\projspace{3}(F))$ containing $x$. ($x$ is the midpoint of an edge joining a point and a hyperplane of $\projspace{3}(F)$.)}
\end{figure}

\begin{cordef}\label{filter:def-filter}
    For $0\leq k\leq\rank\Delta_\ver^=(x)$ we define
    \begin{equation*}
        F_k=\bigcup\{ \sigma\in\simp(\Delta^\geq(x))\mid\height(\sigma)\leq k\}.
    \end{equation*}
    Then $\Delta^>(x)*\Delta_\hor(x)=F_0\subseteq F_1\subseteq\cdots\subseteq F_{\rank\Delta_\ver^=(x)}=\Delta^\geq(x)$ is sequence of subcomplexes. The complex $F_k$ is the disjoint union of the previous complex and the relative stars (in $F_k$) of the height-k-simplices from $\im\restr_\Delta^x$. For any height-k-simplex $\sigma$ from $\im\restr_\Delta^x$ the relative star $\st_{F_k}\sigma$ is the preimage of $\sigma$ under ($\Delta,x$)--restriction and $\lk_{F_k}\sigma=(\lk\sigma)^>(p_\sigma x)*(\lk\sigma)_\hor(p_\sigma x)$.
\end{cordef}

%% file: subcomp.tex
%
%   subcomp.tex
%
%   15.11.2006
%

\subsection{Constructing subcomplexes of $\mathbf{F_k}$}

In order to carry $\dim\Delta$-sphericity from a stage of our filtration to the previous stage we construct $\dim\Delta$-spherical subcomplexes of the previous stage containing the boundaries of the removed relative stars. In this subsection we are going to develop a blueprint for these complexes.

\begin{notation}
    For any pair $\sigma\op\tau$ of opposite simplices from $\Delta$ and any simplex $\theta$ of $\overline{\st\sigma}$
    we put
    \begin{align*}
        \defclpp(\sigma,\theta,\tau)&=\conv(\lambda_\sigma\theta,\proj_\tau\lambda_\sigma\theta),\\
        \defclp(\sigma,\theta,\tau)&=\defclpp(\sigma,\theta,\tau)\setminus\st\sigma,\\
        \defcl(\sigma,\theta,\tau)&=\defclpp(\sigma,\theta,\tau)\setminus(\st\sigma\cup\st\tau).
    \end{align*}
\end{notation}

\begin{lemma}\label{subc:cone-cap-star-boundary}
    Let $\sigma\op\tau$ be opposite simplices of $\Delta$ and let $\theta$ be a simplex of $\overline{\st\sigma}$. Suppose $\defcl^*$ is one of the complexes $\defclpp(\sigma,\theta,\tau)$, $\defclp(\sigma,\theta,\tau)$ or $\defcl(\sigma,\theta,\tau)$. Then $\defcl^*\cap\ \partial\st\sigma=\partial\lambda_\sigma\theta\setminus\st\sigma$.
\end{lemma}
\begin{proof}
    Any face of $\lambda_\sigma\theta$ that is not contained in $\st\sigma$ lies in $\defcl^*$. Conversely, let $\eta$ be a simplex of $\conv(\lambda_\sigma\theta,\proj_\tau\lambda_\sigma\theta)\cap\partial\st\sigma$. Since $\sigma$ is opposite to $\tau$, we have $\lambda_\sigma\theta=\proj_\sigma\proj_\tau\lambda_\sigma\theta$.
    Then $\conv(\lambda_\sigma\theta,\proj_\tau\lambda_\sigma\theta)$ equals $\conv(\sigma,\proj_\tau\lambda_\sigma\theta)$ and
    $\lambda_\sigma\eta$ is a face of $\lambda_\sigma\theta$, because $\proj_\sigma\proj_\tau\lambda_\sigma\theta$ is the unique maximal simplex of $\conv(\sigma,\proj_\tau\lambda_\sigma\theta)\cap\st\sigma$. Hence, $\eta$ is contained in $\partial\lambda_\sigma\theta\setminus\st\sigma$.
\end{proof}

\begin{notation}
    Let $\sigma\op\tau$ be opposite simplices of $\Delta$. For any non empty subcomplex $L$ of $\overline{\st\sigma}$ we put
    \begin{equation*}
        \defclpp(\sigma,L,\tau) = \bigcup\nolimits_{\theta\in\simp(L)}\defclpp(\sigma,\theta,\tau)
    \end{equation*}
    and analogously $\defclp(\sigma,L,\tau)$ as well as $\defcl(\sigma,L,\tau)$.
\end{notation}

\begin{lemma}\label{subc:cone-in-south}
    Let $\sigma$ be a simplex of $\Delta^=(x)$ and let $\tau$ be a simplex opposite to $\sigma$. Then $\defclpp(\sigma,L,\tau)$
    is a subcomplex of $\Delta^\geq(x)$ for any subcomplex $L$
    of $\partial\st_{\Delta^\geq(x)}\sigma$.
\end{lemma}
\begin{proof}
    Let $y\in\sigma$ be a point and let $z$ be its antipode in $\tau$. From the triangle inequality we get $d(z,x)\geq\pi/2$. Let $u$ be a point of $\defclpp(\sigma,L,\tau)\setminus\{y,z\}$. Then $u\in\defclpp(\sigma,\theta,\tau)$ for some simplex $\theta\in\simp(L)$. The projection $p_yu$ of $u$ on $\partial\st\sigma$ is a point of $\partial\lambda_\sigma\theta$ by \ref{subc:cone-cap-star-boundary}. It holds $d(p_yu,x)\geq\pi/2$, because $\lambda_\sigma\theta$ is a simplex of $\Delta^\geq(x)$. Hence, $\angle_y(x,u)$ is not acute and we obtain $d(x,u)\geq\pi/2$ by the spherical law of cosines.
\end{proof}

\begin{lemma}\label{subc:cone-in-filter}
    Let $\sigma\neq\emptyset$ be a simplex of $\im\restr_\Delta^x$ and let $\tau$ be opposite to $\sigma$. Let $L$ be a subcomplex of $\partial\st_{F_{\height(\sigma)}}\sigma$. If $\defclpp(\sigma,L,\tau)\cap\Delta^=(x)$ is contained in $\overline{\st\sigma}$, then $\defclpp(\sigma,L,\tau)$ is a subcomplex of $F_{\height(\sigma)}$ and $\defclp(\sigma,L,\tau)=\defclpp(\sigma,L,\tau)\cap F_{\height(\sigma)-1}$.
\end{lemma}
\begin{proof}
    Let $\theta$ be a simplex of $\defclp(\sigma,L,\tau)$. Then $\theta\in\simp(\Delta^\geq(x))$ by \ref{subc:cone-in-south}. Either $\theta$ is a simplex of $\Delta^>(x)$, which means $\height(\theta)=0<\height(\sigma)$, or we have $\theta_x^=\neq\emptyset$ (see \ref{filter:def-restriction}). In the latter case $\theta_x^=$ is contained in $\partial\lambda_\sigma\eta\setminus\st\sigma$ for some simplex $\eta\in\simp(L)$ by \ref{subc:cone-cap-star-boundary}. Since $\sigma$ is not a face of $\theta_x^=$, we get $\restr_\Delta^x(\theta_x^=)\neq\restr_\Delta^x(\lambda_\sigma\eta)=\sigma$ from \ref{filter:equal-restriction}. Then $\height(\theta)=\height(\theta_x^=)<\height(\lambda_\sigma\eta)=\height(\sigma)$ according to \ref{filter:faceheight}. Therefore $\defclp(\sigma,L,\tau)$ is a subcomplex of $F_{\height(\sigma)-1}$. The claim follows, since any simplex of $\defclpp(\sigma,L,\tau)$ is either a simplex of $\defclp(\sigma,L,\tau)$ or lies in $\st_{F_{\height(\sigma)}}\sigma$ and $F_{\height(\sigma)-1}\subseteq F_{\height(\sigma)}\setminus \st_{F_{\height(\sigma)}}\sigma$ by \ref{filter:def-filter}.
\end{proof}

\noindent Let $\sigma\neq\emptyset$ be a simplex of $\im\restr_\Delta^x$. In an ideal situation, there would be an opposite $\tau$ of $\sigma$ such that $\defclpp(\sigma,\partial\st_{F_{\height(\sigma)}}\sigma,\tau)\cap\Delta^=(x)\subseteq\overline{\st\sigma}$. But if $\lk\sigma$ has a join factor that is contained in the equator complex, such an opposite does not exist in general. We will show that $\defclpp(\sigma,L,\tau)\cap\Delta^=(x)\subseteq\overline{\st\sigma}$, for some opposite $\tau$ of $\sigma$, provided $\Delta$ is thick and $L$ is a subcomplex of $\partial\st_{F_{\height(\sigma)}}\sigma$ such that $L\cap\Delta^=(x)\subseteq\Sigma$ for some apartment $\Sigma$ containing $x$ and $\sigma$. Let us begin by a sketch of the idea:

Let $y$ be a point of $\sigma$ and let $\Sigma$ be an apartment containing $x$ and $y$. Let $u$ be a point of the intersection $\Omega_\Sigma^=(x)\cap\partial\st_\Sigma\sigma$ and let $\Sigma'$ be some other apartment that contains $y$, $u$, and $p_yx$. Further, let $\tau$ be the opposite of $\sigma$ in $\Sigma'$ and denote the antipode of $y$ in $\tau$ by $z$. We denote the simplex carrying $u$ by $\theta$. Finally, we denote the simplex carrying $x$ by $\xi$ and the simplex carrying $p_yx$ by $\chi$.
\begin{figure}[ht]
    \begin{center}
        \includegraphics[scale=0.8]{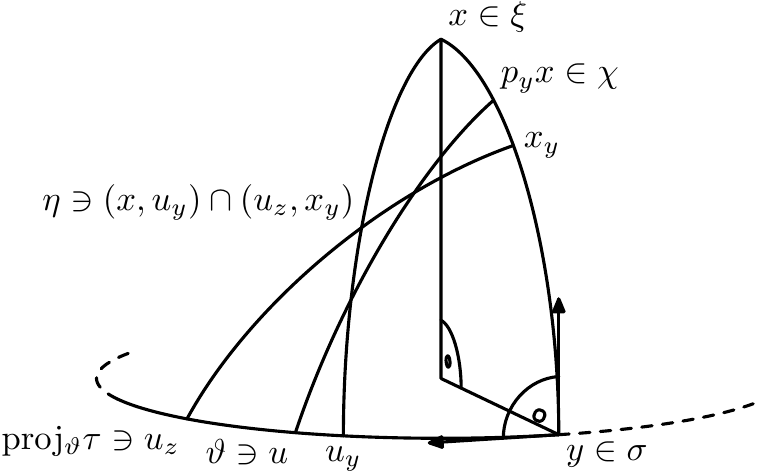}
    \end{center}
    \caption{$d(u_z,x)=\pi/2$ implies $\proj_{\theta\cup\chi}\xi=\eta=\proj_{\theta\cup\chi}\proj_\theta\tau$.}
\end{figure}

Suppose there is a point $u_z\in[z,u)\cap\st\theta$ at distance $\pi/2$ to $x$. By the spherical law of cosines the triangle $(u_z,y,x)$ is contained in some apartment. Then for any $u_y\in(u,y)$ and any $x_y\in(y,p_yx)$ the segments $(x,u_y)$, $(u_z,x_y)$, and $(u,p_yx)$ intersect each other. We are able to choose $u_y$ and $x_y$ such that $(x,u_y)\cap(u_z,x_y)$ is near $(u,p_yx)$ but outside the triangle $(u,y,p_yx)$. Then the simplex $\eta$ carrying $(x,u_y)\cap(u_z,x_y)$ is not contained in $\overline{\st\sigma}$. Since $\proj_\theta\tau$ is the simplex carrying $u_z$, we obtain $\proj_{\theta\cup\chi}\xi=\eta=\proj_{\theta\cup\chi}\proj_\theta\tau=\proj_{\theta\cup\chi}\tau$
by \ref{building:projection}. Therefore $\Sigma\cap\Sigma'$ is not contained in $\overline{\st\sigma}$.

Note, that $d(u_z,y)\neq\pi/2$ for any opposite $\tau$ of $\sigma$ chosen in an apartment $\Sigma'$ satisfying $\Sigma\cap\Sigma'=\overline{\st_\Sigma\sigma}$. Provided $\Delta$ is thick, such an apartment exists by the following lemma, which follows directly from \cite[Proposition~4.1]{abvm2}.

\begin{lemma}\label{subc:apartment-intersection}
    Let $\Delta$ be a thick and let $\Sigma\subseteq\Delta$ be an apartment. If $K\subseteq\Sigma$ is a non empty, convex chamber subcomplex, then there is an apartment $\Sigma'$ such that $K=\Sigma\cap\Sigma'$.
\end{lemma}

\begin{lemma}\label{subc:proj-condition}
    Let $\sigma\neq\emptyset$ be a simplex of $\Delta^=(x)$ and let $\tau$ be opposite to $\sigma$. Let $L$ be a subcomplex of $\partial\st_{\Delta^\geq(x)}\sigma$. If  $\defclpp(\sigma,L,\tau)\cap\Delta^=(x)$ is not contained in $\overline{\st\sigma}$, then $\proj_\theta\tau\in\simp(\Delta^=(x))$ for some simplex $\theta$ of $\Delta^=(x)\cap(\overline{\lambda_\sigma L}\setminus\st\sigma)$.
\end{lemma}
\begin{proof}
    Recall that we are dealing with open simplices. Therefore, a simplex contained in a closed hemisphere lies entirely in the associated equator, if it carries a point of the equator.

    Let $u$ be a point of $\Delta^=(x)\cap\defclpp(\sigma,L,\tau)$ that is not contained in $\overline{\st\sigma}$. Furthermore let $y\in\sigma$ and $z\in\tau$ be antipodal points. If $u=z$, then $\tau=\proj_\emptyset\tau$ is a simplex of $\Delta^=(x)$ by \ref{subc:cone-in-south}. Hence, suppose $u\neq z$.

    The segment $[y,u]$ lies in $\Omega^\leq_\Delta(x)$ as well as in $\defclpp(\sigma,L,\tau)\subseteq\Omega^\geq_\Delta(x)$. Therefore $[y,u]$ is entirely contained in $\Omega^=_\Delta(x)$. By \ref{subc:cone-cap-star-boundary} the simplex $\theta$ carrying $p_yu$ is in $\overline{\lambda_\sigma L}\setminus\st\sigma$ and also a simplex of  $\Delta^=(x)$ by \ref{subc:cone-in-south}. Note, that $p_yu$ lies on the geodesic segment joining $y$, $z$ going through $u$. We obtain $[p_yu,u]\subset[p_yu,z]$, implying by \ref{building:projection}, that $\proj_\theta\tau$ carries a point of $\Omega^=_\Delta(x)$. Then $\proj_\theta\tau$ is a simplex of $\Delta^=(x)$ according to \ref{subc:cone-in-south}.
\end{proof}

\begin{lemma}\label{subc:intersection-condition}
    Let $\sigma\neq\emptyset$ be a simplex of $\im\restr_\Delta^x$ and let $\Sigma$, $\Sigma'$ be apartments such that $x\in\Sigma$ and $\Sigma\cap\Sigma'$ contains $\st_\Sigma\sigma$. Let $\tau$ be the opposite of $\sigma$ in $\Sigma'$. If $\proj_\theta\tau\in\simp(\Delta^=(x))$ for some simplex $\theta$ of $\Sigma^=(x)\cap\partial\st_{F_{\height(\sigma)}}\sigma$, then $\Sigma\cap\Sigma'$ is strictly greater than $\overline{\st_\Sigma\sigma}$.
\end{lemma}
\begin{proof}
    Let $\theta$ be a simplex of $\Sigma^=(x)\cap\partial\st_{F(\height(\sigma))}\sigma$ and $\proj_\theta\tau\in\simp(\Delta^=(x))$. At first we show that if suffices to consider the case $\theta=\emptyset$ and $\tau\in\simp(\Delta^=(x))$.

    We have $\restr_\Delta^x(\sigma\cup\theta)=\sigma$ by \ref{filter:def-filter}, hence $\sigma\setminus\theta$ is a non empty simplex of  $\im\restr_{\lk\theta}^{p_\theta x}$ according to \ref{filter:link}. The star of $\sigma\setminus\theta$ in $\lk_\Sigma\theta$ is contained in the apartment $\lk_{\Sigma'}\theta$ and $p_\theta x$ lies in $\lk_\Sigma\theta$. From \ref{building:projection} we know that $(\proj_\theta\tau)\setminus\theta$ is opposite to $\sigma\setminus\theta$ in $\lk_{\Sigma'}\theta$ and \ref{hc:link} implies that $(\proj_\theta\tau)\setminus\theta$ is in the equator complex of $\lk\theta$ with respect to $p_\theta x$ if and only if $\proj_\theta\tau$ is a simplex of $\Delta^=(x)$.  If $\lk_\Sigma\theta\cap\lk_{\Sigma'}\theta$ is strictly greater than the closed star of $\sigma\setminus\theta$ in $\lk_\Sigma\theta$, then $\Sigma\cap\Sigma'$ is also strictly greater than $\overline{\st_\Sigma\sigma}$. Hence, suppose $\theta=\emptyset$ and $\tau\in\simp(\Delta^=(x))$.

    Let $y\in\sigma$ and $z\in\tau$ be antipodal points. Denote the simplex carrying $x$ by $\xi$ and the simplex carrying $p_yx$ by $\chi$. Since $d(y,x)+d(x,z)=\pi$, there is a geodesic segment $s$ joining $y$ and $z$ going through $x$. It holds $s\setminus\overline{\st\sigma}=[z,p_yx)$. Since $\restr_\Delta^x(\sigma)=\sigma\neq \emptyset$, the pole $x$ can not be a point of $\overline{\st\sigma}$ according to \ref{building:join} and \ref{filter:empty}. Then $x$ is an interior point of $[z,p_yx]$. By \ref{building:projection}, it follows that $\proj_\chi\xi$ is not contained in $\overline{\st\sigma}$ and also that $\proj_\chi\xi=\proj_\chi\tau$ is a simplex of $\Sigma\cap\Sigma'$.
\end{proof}

\begin{cor}\label{subc:apartment-condition}
    Let $\Delta$ be thick and let $\sigma\neq\emptyset$ be simplex of $\im\restr_\Delta^x$. If $L$ is a subcomplex of $\lk_{F_{\height(\sigma)}}\sigma$ such that $L\cap\Delta^=(x)\subseteq\Sigma$ for some apartment $\Sigma$ containing $x$ and $\sigma$, then there is an opposite $\tau$ of $\sigma$ such that $\defclpp(\sigma,L,\tau)$ is a subcomplex of $F_{\height(\sigma)}$ and $\defclp(\sigma,L,\tau)=\defclpp(\sigma,L,\tau)\cap F_{\height(\sigma)-1}$.
\end{cor}
\begin{proof}
    By \ref{subc:apartment-intersection} there is an apartment $\Sigma'$ such that $\Sigma\cap\Sigma'=\overline{\st_\Sigma\sigma}$. Let $\tau$ be the opposite of $\sigma$ in $\Sigma'$. Then for any simplex $\theta$ of $\Sigma^=(x)\cap\partial\st_{F_{\height(\sigma)}}\sigma$, the projection $\proj_\theta\tau$ is not a simplex of $\Delta^=(x)$ by \ref{subc:intersection-condition}. Since $\Delta^=(x)\cap(\overline{\lambda_\sigma L}\setminus\st\sigma)$ is contained in $\Sigma^=(x)\cap\partial\st_{F_{\height(\sigma)}}\sigma$, we obtain $\defclpp(\sigma,L,\tau)\cap\Delta^=(x)\subseteq\overline{\st\sigma}$ by \ref{subc:proj-condition}. Now, the assertion follows from \ref{subc:cone-in-filter}.
\end{proof}

\noindent Since we can not achieve $\defclp(\sigma,\partial\st_{F_{\height(\sigma)}}\sigma,\tau)\subseteq F_{\height(\sigma)-1}$ for a simplex $\sigma$ of $\im\restr_\Delta^x$ that has a join factor in the equator complex, it is necessary to cover the boundary of $\st_{F_{\height(\sigma)}}\sigma$ by subcomplexes. Hence, we need a criterion on $\dim\Delta$--sphericity of an union of cones over subcomplexes of $\partial\st\sigma$.

\begin{definition}\label{subc:def-quasi}
    A subcomplex $\Lambda\subseteq\Delta$ contained in $\Delta\setminus\opp^*(y)$ is quasi-star-shaped with respect to $y$ if and only if for any point $z\in\Lambda$ the segment $[z,p_yz]$ joining $z$ and its geodesic projection on $\partial\st y$ is contained in $\Lambda$. We denote the set of subcomplexes that are quasi-star-shaped with respect to $y$ by $\mathcal{Q}_y$.
\end{definition}

\begin{observation}\label{subc:closed-quasi}
    $\mathcal{Q}_y$ is closed under unions and intersections.
\end{observation}

\begin{observation}\label{subc:quasi-deformation}
    For any $\Lambda\in\mathcal{Q}_y$, there is a strong deformation retraction of $\Lambda$ onto $\Lambda\cap\partial\st y$ induced by the geodesic projection to $\partial\st y$. Specially $\Lambda$ and $\Lambda\cap\partial\st y$ are homotopy equivalent.
\end{observation}

\begin{lemma}\label{subc:cones-are-quasi}
    Let $\sigma\op\tau$ be opposite simplices of $\Delta$ and let $y$ be a point of $\sigma$. Furthermore let $\theta$ be a simplex of $\overline{\st\sigma}$. Then $\defcl(\sigma,\theta,\tau)$ is contained in $\mathcal{Q}_y$.
\end{lemma}
\begin{proof}
    Since $\defclpp(\sigma,\theta,\tau)$ is contained in an apartment, there is only one antipode $z$ of $y$ in $\defclpp(\sigma,\theta,\tau)$. Certainly $z$ is a point of $\tau$. For any point $u\in\defcl(\sigma,\theta,\tau)$, the geodesic segment $s$ joining $y, z$ going through $u$ is contained in $\defclpp(\sigma,\theta,\tau)$, because $s=[y,u]\cup[u,z]$ and $\defclpp(\sigma,\theta,\tau)$ is convex. From \ref{building:projection}, it follows that $s\setminus(\st y\cup\st z)=[p_yu,p_zu]$. Then $[p_yu,u]\subseteq [p_yu,p_zu]$ lies in $\defcl(\sigma,\theta,\tau)$.
\end{proof}

\begin{cor}\label{subc:cones-contractible}
    Let $\sigma\op\tau$ be  opposite simplices of $\Delta$ and let $L$ be a subcomplex of $\lk\sigma$. Then $\defclp(\sigma,L,\tau)$ is ($\dim \sigma+\dim L+1$)-dimensional and contractible.
\end{cor}

\begin{lemma}\label{subc:glue-cones}
    Let $I\neq\emptyset$ be an index set and let $\{L_i\mid i\in I\}$ be a family of non empty subcomplexes of $\lk\sigma$. Furthermore let $\{\tau_i\mid i\in I\}$ be a family of simplices opposite to $\sigma$. For subsets $J\subseteq I$ we put
    \begin{equation*}
        \defclp(J):=\bigcup\nolimits_{j\in J}\defclp(\sigma,L_j,\tau_j)\text{ and }
        \defcl(J):=\bigcup\nolimits_{j\in J}\defcl(\sigma,L_j,\tau_j)
    \end{equation*}
     Suppose $L_i\cap\bigcup_JL_j$ is $\dim\lk\sigma$--spherical, for any $i\in I$ and any non empty, finite $J\subseteq I$. Then $\defclp(I)$ is $\dim\Delta$--spherical.
\end{lemma}
\begin{proof}
    Since it is sufficient to give a proof for $\# I<\infty$, we use induction on $\# I$. The case $\# I =1 $ is clear by \ref{subc:cones-contractible}.

    Now assume $\#I>1$. We put $J=I\setminus\{i\}$ and $J'=\{j\in J\mid \tau_j=\tau_i\}$ for some $i\in I$. Let $y$ be a point of $\sigma$ and let $\alpha:\lk\sigma\rightarrow\lk\tau_i$ denote the isomorphism induced by $\proj_{\tau_i}$. Recall that there is a labeling on $\Delta$. Since all $\tau_j$ have the same labels, $\st\tau_j\cap\st\tau_l$ is empty, unless $\tau_j=\tau_l$. Furthermore $\st\tau_i\cap\defcl(J)$ and $\st\tau_i\cap\defcl(\{i\})$ are empty. We therefore get
    \begin{align*}
        \defclp(\{ i\})\cap\defclp(J)
        &= \left(\defcl(\{i\}\cup\overline{\st_{\defclp(\{i\})}\tau_i}\right)
        \cap\left(\defcl(J)\cup\overline{\st_{\defclp(J')}\tau_i}\right)\\
        &= \left(\defcl(\{i\})\cap\defcl(J)\right)\cup
        \left(\overline{\st_{\defclp(\{i\})}\tau_i}\cap\overline{\st_{\defclp(J')}\tau_i}\right)\\
        &= \left(\defcl(\{i\})\cap\defcl(J)\right)\cup
        \overline{\lambda_{\tau_i}\left(\alpha L_i\cap\bigcup\nolimits_{j\in
        J'}\alpha L_j\right)}
    \end{align*}
    The second complex of this union is contractible, since it is a cone with tip in $\tau_i$. According to \ref{subc:cones-are-quasi} and \ref{subc:closed-quasi}, the first complex is contained in $\mathcal{Q}_y$; and by \ref{subc:quasi-deformation} and \ref{subc:cone-cap-star-boundary}, it is homotopy equivalent to the ($\dim\Delta-1$)--spherical complex
    \begin{align*}
        \partial\st\sigma\cap\defcl(\{ i\})\cap\defcl(J) &=
        (\partial\st\sigma\cap\defcl(\sigma,L_i,\tau_i))\cap
        \bigcup\nolimits_{j\in J}(\partial\st\sigma \cap\defcl(\sigma,L_j,\tau_j))\\
        &= (\partial \sigma*L_i)\cap \bigcup\nolimits_{j\in J} (\partial \sigma*L_j)=
        \partial \sigma*\left(L_i\cap \bigcup\nolimits_{j\in
        J}L_j\right)
    \end{align*}
    Their intersection
    \begin{equation*}
        (\defcl(\{i\})\cap\defcl(J))\cap \overline{\lambda_{\tau_i}\left(\alpha
        L_i\cap\bigcup\nolimits_{j\in J'}\alpha L_j\right)} =\left(\alpha
        L_i\cap\bigcup\nolimits_{j\in J'}\alpha L_j\right)*\partial \tau_i
    \end{equation*}
    is ($\dim\Delta-1$)--spherical as well. Then $\defclp(\{ i\})\cap\defclp(J)$ is ($\dim\Delta-1$)--spherical by \ref{mainconstr}~a). Hence, $\defclp(I)$ is $\dim\Delta$--spherical by the induction hypothesis, \ref{subc:cones-contractible} and again by \ref{mainconstr}~a).
\end{proof}

%% file: sphproof.tex
%
%   sphproof.tex
%
%   22.06.2010
%

\subsection{Proof of Theorem B}

\noindent Now we have got all pieces that are needed to complete the proof of theorem~\ref{theorem-b}.

\begin{prop}
    Open hemisphere complexes of thick, spherical buildings are spherical.
\end{prop}
\begin{proof}
    Suppose $\Delta$ is thick. We use induction on $d=\dim\Delta^=(x)$. If $\Delta^=(x)$ is empty then $\Delta^>(x)=\Delta^\geq(x)$ is $\dim\Delta$--spherical by \ref{hc:closed-hc-main}. Let $d\geq 0$ and suppose, open hemisphere complexes with equator complex of dimension less than $d$ are spherical.

    If $\Delta_\hor(x)\neq\emptyset$, we are done by \ref{hc:join} and the induction hypothesis. Hence, assume $\Delta_\hor(x)$ is empty. We show that for any non empty simplex $\sigma$ of $\im\restr_\Delta^x$ there is a complex $K_\sigma$ fullfilling the following two conditions.
    \begin{itemize}
        \item[]
            Condition 1. $K_\sigma\subseteq F_{\height(\sigma)}$ and
            $K_\sigma=\st_{F_{\height(\sigma)}}\sigma\cup(K_\sigma\cap F_{\height(\sigma)-1})$.
        \item[]
            Condition 2. $K_\sigma\cap F_{\height(\sigma)-1}$ is $\dim\Delta$--spherical.
    \end{itemize}
    \noindent Assuming, we have such $K_\sigma$, we argue as follows: For $1\leq k\leq d+1$, let $I_k$ denote the set of simplices from $\im\restr_\Delta^x$ at height $k$. Since  $\st_{F_k}\sigma\cap\st_{F_k}\tau=\emptyset$, for $\sigma,\tau\in I_k$ with $\sigma\neq\tau$ by \ref{filter:def-filter}, we obtain by the first condition:
    \begin{center}
            $F_k= F_{k-1}\cup\bigcup_{\sigma\in I_k}K_\sigma$ and
            $K_\sigma\cap K_\tau\subseteq F_{k-1}$, for $\sigma,\tau\in I_k$ with $\sigma\neq\tau$.
    \end{center}
    Then $F_{k-1}$ is $\dim\Delta$--spherical provided the
    same holds for $F_k$, by \ref{mainconstr}~b) and the second condition. Recall that $\Delta_\hor(x)$ is empty. Hence, $\Delta^>(x)=F_0$ is $\dim\Delta$--spherical, since $F_{d+1}=\Delta^\geq(x)$ is $\dim\Delta$--spherical by \ref{hc:closed-hc-main}.\medskip

    \noindent It remains to find the complexes $K_\sigma$. Let $\sigma$ be a non empty simplex of $\im\restr_\Delta^x$. We put $L=(\lk\sigma)^>(p_\sigma x)$. There are two cases:
    \medskip

    \noindent$\underline{(\lk\sigma)_\hor(p_\sigma x)=\emptyset}\ $: In this case we have $\lk_{F_{\height(\sigma)}}\sigma=L$ by \ref{filter:def-filter}. Since $L\cap\Delta^=(x)$ is empty, \ref{subc:apartment-condition} provides us with an opposite $\tau\op\sigma$ such that $\defclpp(\sigma,L,\tau)$ is a subcomplex of $F_{\height(\sigma)}$ and $\defclp(\sigma,L,\tau)=\defclpp(\sigma,L,\tau)\cap F_{\height(\sigma)-1}$. Certainly, $\defclpp(\sigma,L,\tau)=\st_{F_{\height(\sigma)}}\sigma\cup\defclp(\sigma,L,\tau)$ since $\st_{F_{\height(\sigma)}}\sigma=\lambda_\sigma L$. Furthermore $\defclp(\sigma,L,\tau)$ is $\dim\Delta$--spherical by \ref{subc:cones-contractible}. Hence, $K_\sigma=\defclpp(\sigma,L,\tau)$ fullfills the two conditions above.
    \medskip

    \noindent$\underline{(\lk\sigma)_\hor(p_\sigma x)\neq\emptyset}\ $: In this case we further put $L_h=(\lk\sigma)_\hor(p_\sigma x)$. Let $C$ be a chamber of $L_h$ and let $\cal{A}$ denote the set of apartments of $L_h$ that contain $C$. For any apartment $A\in\cal{A}$, there is an apartment $\Sigma_A$ of $\Delta$ that contains $x$ and $\overline{\lambda_\sigma A}$.

    Let $C'$ be opposite to $C$ in $A$. Then choose points $y\in C\cup\sigma$ and $y'\in C'\cup\sigma$ and have a look at the triangle $(x,y,y')$. Since $d(x,y)=\pi/2=d(x,y')$ and $\angle_y(x,y')=\pi/2$, equality holds in the spherical law of cosines. Hence, there is an apartment $\Sigma_A$ that contains $x$ and $\overline{\lambda_\sigma A}=\conv(C\cup\sigma,C'\cup\sigma)$.

    Let $A\in\cal{A}$ be arbitrary. Since $(L*A)\cap\Delta^=(x)$ is contained in $\Sigma_A$, we get an opposite $\tau_A\op\sigma$ by \ref{subc:apartment-condition} such that $\defclpp(\sigma,L*A,\tau_A)$ is a subcomplex of $F_{\height(\sigma)}$ and $\defclp(\sigma,L*A,\tau_A)=\defclpp(\sigma,L*A,\tau_A)\cap F_{\height(\sigma)-1}$. We define
    \begin{center}
        $K_\sigma = \bigcup_{A\in\cal{A}}\defclpp(\sigma,L*A,\tau_A)$
         and
        $K'_\sigma = \bigcup_{A\in\cal{A}}\defclp(\sigma,L*A,\tau_A)$.
    \end{center}
    Then $K_\sigma$ is a subcomplex of $F_{\height(\sigma)}$, and $K'_\sigma=K_\sigma\setminus\st\sigma$ is its intersection with $F_{\height(\sigma)-1}$. From \ref{filter:def-filter}, we know that $\lk_{F_{\height(\sigma)}}\sigma=L*L_h$. Since $L_h$ is covered by $\cal{A}$, the link of $\sigma$ in $K_\sigma$ is also $L*L_h$. Therefore, the stars of $\sigma$ in $K_\sigma$ and $F_{\height(\sigma)}$ coincide. Hence, $K_\sigma=\st_{F_{\height(\sigma)}}\sigma\cup K'_\sigma$ fullfills the first condition.

    The open hemisphere complex $L$ is $\dim (\lk\sigma)_\ver(p_\sigma x)$--spherical by the induction hypothesis, since $\dim(\lk\sigma)^=(p_\sigma x)<d$. For any $A\in\cal{A}$ and any non empty, finite $\cal{A}'\subseteq\cal{A}$, the intersection $A\cap\bigcup\cal{A}'$ is an union of convex subcomplexes of $A$ each of which contains $C$. Therefore $A\cap\bigcup\cal{A}'$ equals $A$ or is contractible. Then
    \begin{equation*}
        (L*A)\cap\bigcup\nolimits_{A'\in\cal{A}'}(L*A')=
        L*\left(A\cap\bigcup\cal{A}'\right)
    \end{equation*}
    is $\dim\lk\sigma$--spherical. From \ref{subc:glue-cones} we now get the $\dim\Delta$--sphericity of $K'_\sigma=K_\sigma\cap F_{\height(\sigma)-1}$, hence the second condition.
    \medskip
\end{proof}

\begin{prop}
    Open hemisphere complexes of thick, spherical buildings are non contractible.
\end{prop}
\begin{proof}
    Suppose $\Delta$ is thick. By \ref{hc:join} we further suppose $\Delta_\hor(x)=\emptyset$. We will show the existence of a $\dim\Delta$--sphere in $\Delta^>(x)$ by induction on $\dim\Delta$.

    If $\dim\Delta=0$, then $\Omega^\leq_\Delta(x)=\{x\}$ is a single point. Hence, $\Delta^>(x)=\Delta\setminus\{x\}$ contains a 0--sphere.

    Let $\dim\Delta>0$ and suppose, open hemisphere complexes of dimension less than $\dim\Delta$ contain a topdimensional sphere. If $\Delta^=(x)=\emptyset$, then $\Delta^>(x)$ is a closed, coconvex supported subcomplex and the assertion follows from \ref{hc:closed-hc-main}. We therefore assume $\Delta^=(x)\neq\emptyset$.

    Let $y\in F_1\setminus F_0$ be a vertex. By \ref{filter:def-filter}, its relative link $L=\lk_{F_1}y$ is a subcomplex of $F_0=\Delta^>(x)$ and an open hemisphere complex of $\lk y$. According to the induction hypothesis, there is a ($\dim\Delta-1$)--dimensional sphere $S\subseteq L$. Suppose $y$ has two opposites $z', z''\in\Delta^>(x)$. By \ref{subc:cone-in-filter} and \ref{subc:proj-condition} the union $\defclp(y,L,z')\cup\defclp(y,L,z'')$ is a subcomplex of $\Delta^>(x)$.  This complex contains the two geodesic cones over $S$ with tip in $z'$ and $z''$, hence it contains a $\dim\Delta$--sphere. It remains to show that $y$ has two opposites in $\Delta^>(x)$.

    Let $\Sigma$ be an apartment containing $x$ and $y$. From  \ref{subc:apartment-intersection} we obtain an apartment $\Sigma'$, such that $\Sigma\cap\Sigma'=\overline{\st_\Sigma y}$. Denote by $z'$ the vertex of $\Sigma'$ that is opposite to $y$. According to \ref{subc:intersection-condition}, $z'$ is contained in $\Delta^>(x)$ since the intersection of $\Delta^=(x)$ with $\partial\st_{F_1}y$ is empty and $\proj_\emptyset z'=z'$.  Let $C$ be a chamber of $\st y$ that contains $p_yx$ in its closure. Denote the panel $\proj_{z'}C\setminus\{ z'\}$ by $D$. By \ref{subc:apartment-intersection} there is an apartment $\Sigma''$ such that $\Sigma'\cap\Sigma''=\conv(C,D)$. Denote by $z''$ the vertex of $\Sigma''$ that is opposite to $y$. Then $z''\neq z$, since the vertices of  $\conv(C,D)$ are not opposite to $y$. We show that $z''\in\Delta^>(x)$.
    \begin{figure}[ht]
        \begin{center}
            \includegraphics[scale=0.8]{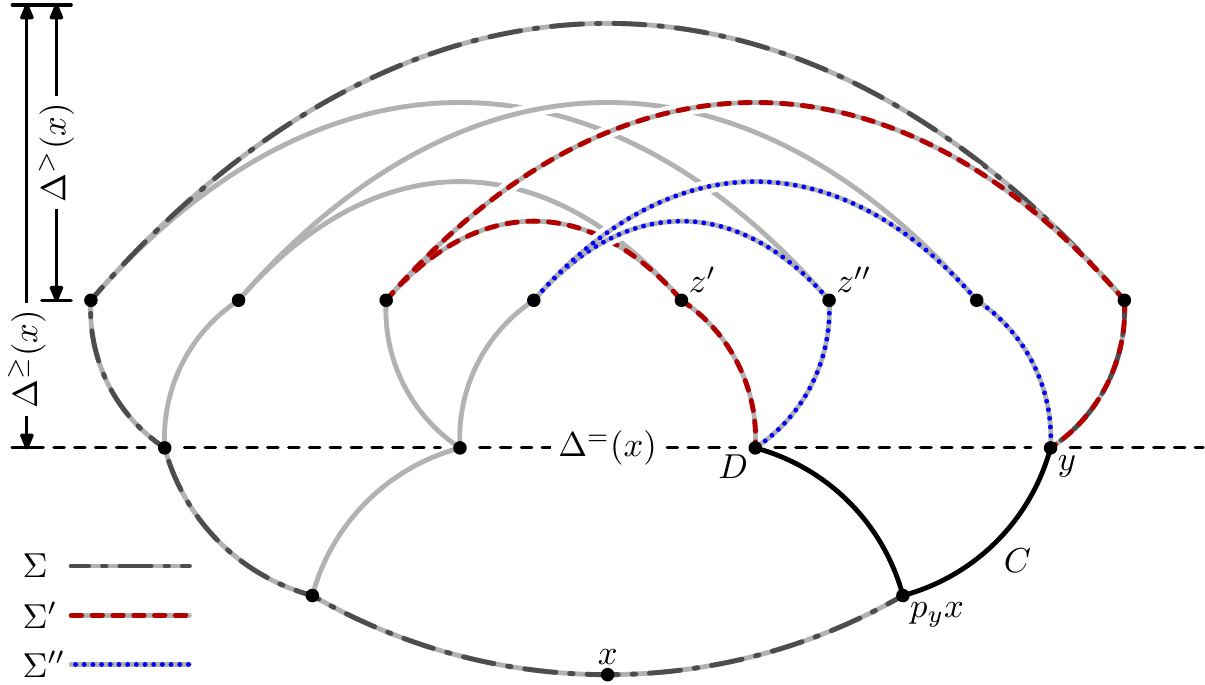}
        \end{center}
    \caption{The construction of antipodes in $\Delta^>(x)$.}
    \end{figure}

    By the triangle inequality we have $d(x,z'')\geq\pi/2$. Suppose $d(x,z'')=\pi/2$. By \ref{building:segments} there is a geodesic  segment joining $y$ and $z''$ going through $x$. Then $x$ is a point of $[y,p_{z''}x]$ because $x$ can not be contained in $\st z''$ by \ref{building:edgelength}. Since $p_yx$ is a point of this segment, we get $p_{z''}x=p_{z''}p_yx$. Observe that $p_{z'}p_yx$ is a point of $\conv(C,\proj_{z'}C)\cap\partial\st z'=\overline{D}\in\simp(\Sigma'')$. The retraction $\rho_{\Sigma'',C}$ on $\Sigma''$ centered at $C$ maps $p_{z'}p_yx$ to $p_{z''}p_yx$, hence $p_{z'}p_yx=p_{z''}p_yx$. Therefore $x$ lies on $[y,p_{z'}p_yx]\subseteq\Sigma'$. This implies $d(x,z')=\pi-d(x,y)=\pi/2$ by \ref{building:segments} in contradiction to $z'\in\Delta^>(x)$.
\end{proof}

%% file: address.tex
\medskip
\begin{flushleft}
    \textsl{Bernd Schulz, Blumenstra{\ss}e 21, 63069 Offenbach, Germany\\
    Email: bernd.schulz@ivir.de}
\end{flushleft}